\documentclass[global]{svjour}%
\usepackage{amssymb}
\usepackage{amsfonts}
\usepackage{amsmath}%
\usepackage{graphicx}
\providecommand{\U}[1]{\protect\rule{.1in}{.1in}}
\begin{document}

\journalname{ }
\title{Optimality conditions in convex multiobjective SIP\thanks{This research was
partially cosponsored by the Ministry of Economy and Competitiveness (MINECO)
of Spain, and by the European Regional Development Fund (ERDF) of the European
Commission, Project MTM2014-59179-C2-1-P.}}
\author{Miguel A. Goberna\inst{1}
\and Nader Kanzi\inst{2}}
\institute{Department of Mathematics, University of Alicante, Alicante, Spain.
%TCIMACRO{\TeXButton{email}{\email{mgoberna@ua.es}}}%
%BeginExpansion
\email{mgoberna@ua.es}%
%EndExpansion
\and Department of Mathematics, Payame Noor University (PNU), Tehran, Iran.
%TCIMACRO{\TeXButton{email}{\email{ nad.kanzi@gmail.com}}}%
%BeginExpansion
\email{ nad.kanzi@gmail.com}%
%EndExpansion
}
\maketitle

\begin{abstract}
The purpose of this paper is to characterize the weak efficient solutions, the
efficient solutions, and the isolated efficient solutions of a given vector
optimization problem with finitely many convex objective functions and
infinitely many convex constraints. To do this, we introduce new and already
known data qualifications (conditions involving the constraints and/or the
objectives) in order to get optimality conditions which are expressed in terms
of either Karusk-Kuhn-Tucker multipliers or a new gap function associated with
the given problem.

\end{abstract}

\section{Introduction}

An optimization problem%
\begin{equation}
(P)\ \ \ \ \text{minimize}\ f(x):=\left(  f_{1}(x),\ldots,f_{p}(x)\right)
\ \text{subject to}\ g_{t}(x)\leq0,\ t\in T,\label{P}%
\end{equation}
with real-valued \emph{objective functions} $f_{i}:\mathbb{R}^{n}%
\longrightarrow\mathbb{R},$ $i\in I:=\{1,\ldots,p\},$ and extended real-valued
constraint functions $g_{t}:\mathbb{R}^{n}\longrightarrow\overline{\mathbb{R}%
}:=[-\infty,+\infty]$, $t\in T$, is called \emph{scalar} when $p=1$ and
\emph{multiobjective} when $p\geq2,$ it is called \emph{ordinary} when $T$ is
finite and \emph{semi-infinite} otherwise, and it is called \emph{convex}
(\emph{linear}, etc.) when all the involved functions (called the \emph{data})
are convex (resp. linear, etc.). The Euclidean spaces $\mathbb{R}^{n}$ and
$\mathbb{R}^{p}$ are called \emph{decision space} and \emph{objective} (or
\emph{outcome}) \emph{space}, respectively.

The objective of this paper is to characterize different types of solutions
for a consistent convex multiobjective semi-infinite programming (MOSIP in
brief), i.e., we consider a problem $(P)$ as in (\ref{P}) with $T$ infinite
(not necessarily equipped with some topology), $p\geq2,$ and convex data. This
type of problem arises in a natural way in robust convex multiobjective
programming when at least one of the uncertainty sets for the constraints is
infinite (see \cite{KL14} and references therein). We also assume that all
constraint functions are lower semicontinuous (lsc for short). So, the
\emph{feasible set} of $(P)$, denoted by $S$, is a nonempty closed convex set.
The characterization of the solutions of $\left(  P\right)  $ in this paper
are expressed either in terms of multipliers or in terms of an associated
scalar function called \emph{gap function}.

Many works have been published on optimality conditions in ordinary
multiobjective convex programming and in scalar semi-infinite programming via
\emph{data qualifications} (DQs in short), called \emph{constraint
qualifications} (CQs) and \emph{objective qualifications} (OQs) when they
exclusively involve the constraints and the objectives, respectively (see,
e.g.,\cite{HirLem}, \cite{GorLop}, \cite{LiNah}, \cite{MORDNG}, etc.). The
closest antecedents for this paper are the following works dealing with
Karush-Kuhn-Tucker (KKT in short) type optimality conditions for different
classes of multiobjective semi-infinite programming problems:

\begin{itemize}
\item \cite{GOVATO}, on linear MOSIP problems.

\item \cite{GOVATO2}, where the data are convex and satisfy the assumption
that $(P)$ is \emph{continuous}, meaning that $T$ is a compact Hausdorff
topological space and the mapping $\left(  t,x\right)  \longmapsto g_{t}(x)$
is continuous (and so finite-valued) on $T\times\mathbb{R}^{n}$. This is the
case whenever $T$ is a finite set equipped with the discrete topology and the
constraint functions are finite-valued. The continuity of $(P)$ is a strong
assumption allowing to get KKT conditions via the linearization of the data.

\item \cite{GURU}, where the data are differentiable.

\item \cite{CHUKIM}, \cite{CHUYAO}, \cite{KAN2}, \cite{KAN1}, and \cite{KANO},
where the data are locally Lipschitz on $\mathbb{R}^{n},$ assumption which
allows to apply the non-smooth analysis machinery.
\end{itemize}

Observe that a proper convex function is subdifferentiable (but not
necessarily differentiable) at the interior of its domain, while it is also
locally Lipschitz on that set under some additional assumption (e.g., when it
is bounded above on some nonempty open subset of its domain \cite[Corollary
2.2.13]{Zalinescu}). So, the class of MOSIP problems considered in
\cite{GOVATO} is a subclass of those considered in \cite{GURU}, \cite{CHUKIM},
\cite{CHUYAO}, \cite{KAN2}, \cite{KAN1}, and \cite{KANO}, while \cite{GOVATO2}
does not, except in particular cases (e.g., when all data are bounded above on
$\mathbb{R}^{n}$). Notice also that the class of MOSIP problems in
\cite{GOVATO2} is much smaller than the one considered in this paper. In
conclusion, to the best of the authors' knowledge, this is the first paper
providing optimality conditions for non-continuous convex MOSIP problems.

Concerning the second tool used in this paper to get optimality conditions,
let us recall that the term `gap function' was the name given in \cite{AUS76}
to the objective function of the reformulation of certain variational
inequality as a scalar optimization unconstrained problem. By extension, this
is the name usually given to those functions allowing to reformulate a given
optimization problem as a simpler one. Gap functions have been proposed for
scalar optimization problems with smooth convex data \cite{HEA}, ordinary
multiobjective problems with differentiable data \cite{CHGOYA}, ordinary
multiobjective problems with nonsmooth data (\cite{SARSOLEYMAN}, \cite{Sol},
\cite{SOLEY}), scalar semi-infinite problems with quasiconvex data
\cite{KanziSoleymani}, etc. In this paper we introduce a new gap function in
order to get optimality conditions in convex MOSIP.

We analyze in this paper three types of optimal solutions for $(P),$ namely:

\begin{itemize}
\item Weak efficient solutions, which are commonly considered those solutions
that can be most easily computed. A recent result on vector optimization
problems (where the decision and the objective spaces are possibly infinite
dimensional) proved in \cite{DIGOLOLO}, will be used to characterize weak
solutions of $\left(  P\right)  $ in some particular cases.

\item Efficient solutions, which are the most used in practice.

\item Isolated efficient solutions, which are the most stable under small
perturbations of the objective functions, as we comment after its definition below.
\end{itemize}

The definitions of efficient and weak efficient solutions involve orderings in
$\mathbb{R}^{p}$ induced by its positive cone $\mathbb{R}_{+}^{p}$\ and by its
interior $\mathbb{R}_{++}^{p}.$\ Indeed,\ given $x,y\in\mathbb{R}^{p} $, we
write $x\leqq y\ $(resp.$\ x<y$) when $x_{i}\leq y_{i}\ $(resp.$\ x_{i}<y_{i}%
$) for all $i\in I$. Moreover, we write $x\leq y$ when $x\leqq y$ and $x\neq
y$.

An element $\widehat{x}\in S$ is said to be an \emph{efficient solution}
(resp. a \emph{weak efficient solution}) for $(P)$ if there is no $x\in S$
satisfying $f(x)\leq f(\widehat{x})$ (resp., $f(x)<f(\widehat{x})$). Moreover,
following \cite{GGR}, we say that an element $\widehat{x}\in S$ is an
\emph{isolated efficient} solution for $(P)$ if there exists a constant
$\nu>0$ such that
\[
\max_{i\in I}\{f_{i}(x)-f_{i}(\widehat{x})\}\geq\nu||x-\widehat{x}||,\text{
}\forall x\in S.
\]
The weak efficient solutions and the efficient solutions have been
characterized in \cite{GOVATO2} for convex continuous MOSIP problems, while it
seems that no previous work deals with isolated efficient solution in the
convex setting. Obviously, any isolated efficient solution is an efficient
solution (actually the unique efficient solution) and any efficient solution
is a weak efficient solution. The concept of isolated efficient solution is an
extension, to $p>1,$ of the concept of isolated local minimum introduced in
\cite{AUS84} to study the stability of scalar optimization problems. The
isolated minima are called strongly unique optimal solutions in the linear
semi-infinite setting, where they also play a crucial role in the sensitivity
analysis with respect to perturbations of the objective function
(\cite{GOGOGUTO}, \cite{GOTETO}). We now sketch the role that isolated
efficient solutions could play in the MOSIP setting.

Assume that $\widehat{x}$ is an isolated efficient solution of $\left(
P\right)  $ with associated constant $\nu.$\ If we perturb the objective
function $f_{i}$ with a function $p_{i}:S\longrightarrow\mathbb{R}$ which is
locally Lipschitz at $\widehat{x}$ with constant $L_{i}<$\ $\nu,$ $i\in I,$
and define $L:=\max_{i\in I}L_{i}<\nu,$ then we have
\[%
\begin{array}
[c]{ll}%
\max_{i\in I}\{\left(  f_{i}+p_{i}\right)  (x)-\left(  f_{i}+p_{i}\right)
(\widehat{x})\} & \geq\max_{i\in I}\{\left[  f_{i}(x)-f_{i}(\widehat
{x})\right]  -L_{i}||x-\widehat{x}||\}\\
& \geq\max_{i\in I}\{\left[  f_{i}(x)-f_{i}(\widehat{x})\right]
\}-L||x-\widehat{x}||\\
& \geq\left(  \nu-L\right)  ||x-\widehat{x}||
\end{array}
\]
for any $x\in S,$ so that $\widehat{x}$ is still an isolated efficient
solution of the perturbed MOSIP problem%
\[
\ \text{minimize}\ \left(  f_{1}(x)+p_{1}\left(  x\right)  ,\ldots
,f_{p}(x)+p_{p}\left(  x\right)  \right)  \ \text{subject to}\ g_{t}%
(x)\leq0,\ t\in T.
\]
In the particular case of linear perturbations of the form $p_{i}\left(
x\right)  =c_{i}^{\prime}x,$ with $c_{i}\in\mathbb{R}^{n},$ $i\in I,$ we can
consider the following parametric convex MOSIP problem,
\[
(P_{C})\ \ \ \ \text{minimize}\ f(x)+Cx\ \text{subject to}\ g_{t}%
(x)\leq0,\ t\in T,
\]
whose parameter is the matrix $C:=\left[  c_{1}\mid...\mid c_{p}\right]
^{\prime}\in\mathcal{M}_{p\times n},$ the real linear space of $p\times n$
matrices equipped with any norm and null element $0_{p\times n}$ (the zero
$p\times n$ matrix). In the same way that sensitivity analysis in scalar
optimization deals with the optimal value function of a parametric (scalar)
problem, the sensitivity analysis of $(P_{C})$ is focused on the multifunction
$\mathcal{F}:\mathcal{M}_{p\times n}\rightrightarrows\mathbb{R}^{p}$
associating to each $C\in\mathcal{M}_{p\times n}$ the image of the efficient
set of $(P_{C})$ by its objective function. If $\widehat{x}$ is an isolated
efficient solution of $(P),$ then there exists a neighborhood $\mathcal{N}$ of
$0_{p\times n}$ in $\mathcal{M}_{p\times n}$ where $\max_{i\in I}\left\Vert
c_{i}\right\Vert <\nu$ for all $C=\left[  c_{1}\mid...\mid c_{p}\right]
^{\prime}\in\mathcal{N}$. Since $\widehat{x}\ $is the unique efficient
solution of $(P_{C})$ for any $C\in\mathcal{N},$ the restriction of
$\mathcal{F}$ to $\mathcal{N}$ is the affine function $\mathcal{F}\left(
C\right)  =f(\widehat{x})+C\widehat{x}.$

The paper is organized as follows. Section 2 introduces the necessary
notations and concepts, and analyzes the relationships between the data
qualifications used in this paper and other well known ones. The remaining
three sections provide different characterizations of weakly efficient
solutions (Section 3), efficient solutions (Sections 4), and isolated
efficient solutions (Section 5) via KKT-type conditions, gap function, and
continuous linear functionals on the constraint space (only in Section 3).

\section{Data qualifications}

Our notation and terminology are basically standard. Given a locally convex
Hausdorff topological vector space $Z$ and a set $\emptyset\neq M\subseteq Z
$, the \emph{closure} of $M$, the \emph{interior} of $M,$ the \emph{relative
interior} of $M,$\ the \emph{strong quasi-relative interior} of $M,$ the
boundary of $M,$ the \emph{linear span} of $M,$ the \emph{convex hull} of $M$,
and the \emph{convex cone} (containing the origin) generated by $M$ are
respectively denoted by $\overline{M}$, $\operatorname*{int}(M),$
$\operatorname*{ri}(M), $\ $\operatorname*{sqri}\left(  M\right)  ,$
$\operatorname*{bd}\left(  M\right)  ,$ $\operatorname*{span}\left(  M\right)
,$ $\operatorname*{conv}(M)$, and $\operatorname*{cone}(M)$. The null vector
of $Z$ is denoted by $0_{Z}$ and the topological dual of $Z $ is denoted by
$Z^{\ast},$ always equipped with the weak$^{\ast}$ topology. The duality
product of $z^{\ast}\in Z^{\ast}$ by $z\in Z$ is denoted by $\left\langle
z^{\ast},z\right\rangle .$ In this paper we handle three infinite dimensional
spaces related with the index set $T:$

\begin{itemize}
\item $\mathbb{R}^{T}$ is the space of real-valued functions on $T$ equipped
with the product topology. The \emph{support} of $\lambda\in\mathbb{R}^{T}$ is
$\left\{  t\in T\mid\lambda_{t}:=\lambda\left(  t\right)  \neq0\right\}  . $

\item $l_{\infty}\left(  T\right)  $ is the Banach space of all bounded
functions from $T$ to $\mathbb{R}$ equipped with the supremum norm.

\item When $T$ is a compact topological space, $\mathcal{C}\left(  T\right)  $
represents the subspace of $l_{\infty}\left(  T\right)  $ formed by the
continuous functions.
\end{itemize}

When $Z=\mathbb{R}^{n},$ the duality product $\left\langle x,y\right\rangle $
coincides with the \emph{scalar product} $x^{\prime}y,$ for $x,y\in
\mathbb{R}^{n}$, $||x||=\sqrt{x^{\prime}x}$ denotes the \emph{Euclidean norm}
of $x, $ $\mathbb{B}_{n}$ represents the \emph{closed unit ball} and $0_{n}$
stands for the null vector $0_{n}:=0_{\mathbb{R}^{n}}.$ Given $A,B\subseteq
\mathbb{R}^{n},$ one says that $A$ is \emph{closed regarding} $B$ if
$B\cap\overline{A}=B\cap A.$

The \emph{negative polar cone} and the \emph{strictly negative polar cone} of
$\emptyset\neq M\subset\mathbb{R}^{n}$ (not necessarily a convex cone) are
defined respectively by $M^{0}:=\{d\in\mathbb{R}^{n}\mid x^{\prime}%
d\leq0,\ \forall x\in M\},$ and $M^{-}:=\{d\in\mathbb{R}^{n}\mid x^{\prime
}d<0,\ \forall x\in M\}.$ The bipolar theorem (see, e.g., \cite[Proposition
III.4.2.7]{HirLem}) states that $(M^{0})^{0}=\overline{\operatorname*{cone}%
}(M).$ Also, it is well-known that if $M^{-}\neq\emptyset$, then
$\overline{M^{-}}=M^{0}$.

Regarding extended functions, given $h:\mathbb{R}^{n}\longrightarrow
\overline{\mathbb{R}}:=\mathbb{R\cup}\left\{  \pm\infty\right\}  ,$
$\operatorname{dom}h$ denotes the domain of $h$\ and, for $x\in h^{-1}\left(
\mathbb{R}\right)  ,$ $h^{\prime}\left(  x;d\right)  \ $and $\partial h\left(
x\right)  $ denote the \emph{directional derivative} of $h$ at $x$ in the
direction of $d\in\mathbb{R}^{n}$ and the \emph{Fenchel-Moreau
subdifferential} of $h$ at $x,$ respectively. By $\Gamma\left(  \mathbb{R}%
^{n}\right)  $ we denote the class of all lsc, convex, and proper functions
from $\mathbb{R}^{n}$ to $\overline{\mathbb{R}}.$

Given $\widehat{x}\in M\subseteq\mathbb{R}^{n}$, we denote the
\emph{contingent cone} (also called \emph{Bouligand tangent cone}) and the
\emph{feasible direction cone} of $M$ at $\widehat{x},$ respectively, by
$C(M,\widehat{x})$ and by $D(M,\widehat{x})$, i.e.,
\[
C(M,\widehat{x}):=\big\{v\in\mathbb{R}^{n}\mid\exists t_{r}\downarrow
0,\ \exists v_{r}\rightarrow v\ :\ \widehat{x}+t_{r}v_{r}\in M,\ \forall
r\in\mathbb{N}\big\},
\]
and
\[
D(M,\widehat{x}):=\big\{d\in\mathbb{R}^{n}\mid\exists\delta>0\ :\ \widehat
{x}+\varepsilon d\in M,\ \forall\varepsilon\in(0,\delta)\big\}.
\]
If $M$ is a closed convex set, then $\overline{D(M,\widehat{x})}%
=C(M,\widehat{x})$ for all $\widehat{x}\in M$ \cite[Proposition III.5.2.1]%
{HirLem}. Then, the \emph{normal cone} to $M$ at $\widehat{x}$ is
$N(M,\widehat{x}):=D^{0}(M,\widehat{x})=C^{0}(M,\widehat{x})$.

We associate with $(P)$ the following three functions:

\begin{itemize}
\item The \emph{supremum} (or \emph{marginal})\emph{\ function }%
$\psi:\mathbb{R}^{n}\longrightarrow\overline{\mathbb{R}}$ such that
\begin{equation}
\psi(x):=\sup_{t\in T}g_{t}(x).\label{2.3}%
\end{equation}
This lsc proper convex function allows to reformulate $\left(  P\right)  $ as
an ordinary multiobjective programming problem with a single constraint:
\[
(P_{1})\ \ \ \text{\ minimize}\ f(x):=\left(  f_{1}(x),\ldots,f_{p}(x)\right)
\ \text{subject to}\ \psi(x)\leq0.
\]

\item The \emph{infimum function }$\iota:\mathbb{R}^{n}\longrightarrow
\overline{\mathbb{R}}$ such that
\begin{equation}
\iota(x):=\inf_{t\in T}g_{t}(x).\label{2.4}%
\end{equation}

\item The \emph{gap function} $\vartheta:$ $\mathbb{R}^{n}\times\prod
_{i=1}^{p}\partial f_{i}\left(  x\right)  \times\Delta_{+}^{p}\longrightarrow
\overline{\mathbb{R}},$ such that
\[
\vartheta(x,\xi,\lambda):=\sup_{y\in S}\Big\{\sum_{i=1}^{p}\lambda_{i}\xi
_{i}^{\prime}(x-y)\Big\},
\]
with $\Delta_{+}^{p}:=\left\{  \left(  \alpha_{1},...,\alpha_{p}\right)
\in\mathbb{R}_{+}^{p}\mid\sum_{i=1}^{p}\alpha_{i}=1\right\}  .$ In the
particular case that $p=1$, $\lambda=1$, and $f$ is a convex differentiable
function, $\vartheta$ coincides with the gap function introduced in \cite{HEA}.
\end{itemize}

We also associate with $(P)$ and $\widehat{x}\in S$ the following sets:

\begin{itemize}
\item The \emph{sublevel sets }of $\widehat{x}$ w.r.t. the
objectives,\emph{\ }
\[
Q^{i}(\widehat{x}):=\Big\{x\in S\mid f_{l}(x)\leq f_{l}(\widehat{x}),\ \forall
l\in I\setminus\{i\}\Big\},\text{ }i\in I,
\]
with the convention that $Q^{1}(\widehat{x})=S$ in the scalar case ($p=1$).

\item The \emph{set of }$\varepsilon-$\emph{active indices }at $\widehat{x}, $
with $\varepsilon>0,$\
\[
T_{\varepsilon}(\widehat{x}):=\{t\in T\mid-\varepsilon\leq g_{t}(\widehat
{x})\}.
\]

\item The set of \emph{active indices} at $\widehat{x}$, $T(\widehat
{x}):=T_{0}(\widehat{x}),$ i.e., $T(\widehat{x})=\{t\in T\mid g_{t}%
(\widehat{x})=0\}.$

\item The convex hull of the set of subgradients at $\widehat{x}$ of the
objective functions:
\[
F_{\ast}(\widehat{x}):=\operatorname*{conv}(F(\widehat{x})),\text{ where
}F(\widehat{x}):=\bigcup_{i=1}^{p}\partial f_{i}(\widehat{x}).
\]

\item The convex conical hull of the set of subgradients at $\widehat{x}$ of
the active constraints at $\widehat{x}:$
\[
G_{\ast}(\widehat{x}):=\operatorname*{cone}\big(G(\widehat{x})\big),\text{
where }G(\widehat{x}):=\bigcup_{t\in T(\widehat{x})}\partial g_{t}(\widehat
{x}).
\]
If $t\in T(\widehat{x}),$ $\xi\in\partial g_{t}(\widehat{x}),$ and $d\in
D(S,\widehat{x}),$ there exists $\delta>0$ such that $\widehat{x}+\delta d\in
S$ and, so, $0\geq g_{t}\left(  \widehat{x}+\delta d\right)  \geq\delta
\xi^{\prime}d.$ Since $\xi^{\prime}d\leq0$ for all $d\in D(S,\widehat{x}),$
$\xi\in D^{0}(S,\widehat{x})=N(S,\widehat{x}).$ Thus, $G(\widehat{x})\subseteq
N(S,\widehat{x})$ and taking convex hulls in both sides we get%
\begin{equation}
G_{\ast}(\widehat{x})\subseteq N(S,\widehat{x}).\label{2.1}%
\end{equation}

\end{itemize}

We consider in this paper two global CQs (i.e., CQs not involving $\widehat
{x}\in S$), nine local CQs, two local \emph{data qualifications} (DQs in
brief, i.e., conditions involving all data), and one objective qualification
(OQ)\ for our given convex MOSIP problem $\left(  P\right)  .$ Most of these
conditions are introduced (or at least inspired in conditions used) in works
mentioned in Section 1:

\begin{itemize}
\item The \emph{Slater CQ} (SCQ, \cite{GOVATO2}) holds when there is a
$x_{0}\in\mathbb{R}^{n}$ (called \emph{Slater point}) such that $g_{t}%
(x_{0})<0$ for all $t\in T$.

\item The \emph{strong Slater CQ} (SSCQ, \cite{GOVATO}) holds when there exist
$x_{0}\in\mathbb{R}^{n}$ (called \emph{strong Slater point}) and
$\varepsilon>0$ (\emph{slack}) such that $g_{t}(x_{0})\leq-\varepsilon,$ for
all $t\in T.$ Obviously, $\left(  P\right)  $ satisfies SSCQ if and only if
$(P_{1})$ satisfies SCQ.

\item The \emph{Mangasarian-Fromovitz CQ} (MFCQ, \cite{JOTWWE}) holds at
$\widehat{x}$ if $G(\widehat{x})\neq\emptyset$ and $G^{-}(\widehat{x}%
)\neq\emptyset,$ otherwise (this property is also called local Slater CQ in
\cite{GOVATO2}).

\item The \emph{perturbed Mangasarian-Fromovitz CQ} (PMFCQ, \cite{MORDNG})
holds at $\widehat{x}\in S$ if there exists $x_{\ast}\in\mathbb{R}^{n}$ such
that
\[
\inf_{\varepsilon>0}\sup\left\{  \xi^{\prime}x_{\ast}\mid\xi\in%
%TCIMACRO{\dbigcup \limits_{t\in T_{\varepsilon}(\widehat{x})}}%
%BeginExpansion
{\displaystyle\bigcup\limits_{t\in T_{\varepsilon}(\widehat{x})}}
%EndExpansion
\partial g_{t}(\widehat{x})\right\}  <0.
\]
(This CQ has been recently introduced by Mordukhovich and Nghia in the
framework of non-continuous smooth scalar semi-infinite programming.)

\item The \emph{local Farkas-Minkowski CQ} (LFMCQ, \cite{GorLop}) holds at
$\widehat{x}\in S$ when $N(S,\widehat{x})=G_{\ast}(\widehat{x}).$

\item The \emph{Cottle CQ} (COCQ) holds at $\hat{x}$ when $\{d\in
\mathbb{R}^{n}\mid\psi^{\prime}(\hat{x};d)<0\}\neq\emptyset.$\newline

\item The \emph{Kuhn-Tucker CQ} (KTCQ, \cite{GOVATO2}) at $\widehat{x}\in S$
when
\[
\{d\in\mathbb{R}^{n}\mid\psi^{\prime}(\widehat{x};d)\leq0\}\subseteq
C(S,\widehat{x}).
\]
(Since $S$ is convex,\ we have replaced, at the definition of KTCQ given in
\cite[page 35]{GOVATO2}, the cone of attainable directions of $S$ at
$\widehat{x}$ \ by $C(S,\widehat{x}).$)

\item The \emph{Pshenichnyi-Levin-Valadier CQ} (PLVCQ, \cite{LiNah}) holds at
$\widehat{x}$ when $\partial\psi(\widehat{x})\subseteq G_{\ast}(\widehat{x}).$

\item The \emph{closed cone CQ} (CCCQ) holds at $\widehat{x}$ when $G_{\ast
}(\widehat{x})$ is closed.

\item The \emph{Abadie CQ} (ACQ, \cite{LiNah}) holds at $\widehat{x}\in S$
when $G(\widehat{x})\neq\emptyset$ and $G^{0}(\widehat{x})\subseteq
C(S,\widehat{x}).$

\item The \emph{weak Abadie DQ} (WADQ, \cite{GorLop}) holds at $\widehat{x}\in
S$ when $G(\widehat{x})\neq\emptyset$ and $F^{-}(\widehat{x})\cap
G^{0}(\widehat{x})\subseteq C(S,\widehat{x}).$

\item The\textit{\ }\emph{extended Abadie DQ} (EADQ, \cite{MAEDA}) holds at
$\widehat{x}$ when $G(\widehat{x})\neq\emptyset$ and
\[
F^{0}(\widehat{x})\cap G^{0}(\widehat{x})\subseteq\bigcap_{i=1}^{p}%
C(Q^{i}(\widehat{x}),\widehat{x}).
\]

\item The\textit{\ }\emph{Maeda OQ }(MOQ, \cite{MAEDA}) holds at $\widehat{x}
$ when%
\[
\ F^{0}(\widehat{x})\subseteq\{0_{n}\}\cup\bigcup_{i=i}^{p}\big(\partial
f_{i}(\widehat{x})\big)^{-}.
\]

\end{itemize}

The next result gives a checkable condition equivalent to the \textit{MOQ.}

\begin{proposition}
The MOQ holds at $\widehat{x}\in S$ if and only if $\operatorname*{span}%
\left(  F(\widehat{x})\right)  =\mathbb{R}^{n}.$
\end{proposition}

\begin{proof}
Obviously, $\operatorname*{span}\left(  F(\widehat{x})\right)  =\mathbb{R}%
^{n}$ if and only if $\operatorname*{cone}\left(  F(\widehat{x})\cup\left(
-F(\widehat{x})\right)  \right)  =\mathbb{R}^{n}$ if and only if
$\overline{\operatorname*{cone}}\left(  F(\widehat{x})\cup\left(
-F(\widehat{x})\right)  \right)  =\mathbb{R}^{n}$ or, equivalently, by the
Farkas lemma for semi-infinite systems \cite[Corollary 3.1.3]{GorLop},
\[
\left\{  x\in\mathbb{R}^{n}\mid\xi^{\prime}x\leq0,\forall\xi\in F(\widehat
{x})\cup\left(  -F(\widehat{x})\right)  \right\}  =\left\{  0_{n}\right\}  ,
\]
i.e.,
\begin{equation}
F^{0}(\widehat{x})\cap\big(\left\{  x\in\mathbb{R}^{n}\mid\xi^{\prime}%
x\geq0,\forall\xi\in F(\widehat{x})\right\}  \diagdown\left\{  0_{n}\right\}
\big)=\emptyset.\label{2.6}%
\end{equation}
We finally reformulate (\ref{2.6}) as%
\[%
\begin{array}
[c]{ll}%
F^{0}(\widehat{x}) & \subseteq\mathbb{R}^{n}\diagdown\big(\left\{
x\in\mathbb{R}^{n}\mid\xi^{\prime}x\geq0,\forall\xi\in F(\widehat{x})\right\}
\diagdown\left\{  0_{n}\right\}  \big)\\
& \subseteq\left\{  0_{n}\right\}  \cup\left[  \mathbb{R}^{n}\diagdown
\big(\left\{  x\in\mathbb{R}^{n}\mid\xi^{\prime}x\geq0,\forall\xi\in
F(\widehat{x})\right\}  \big)\right] \\
& =\left\{  0_{n}\right\}  \cup\bigcup_{i=i}^{p}\big(\partial f_{i}%
(\widehat{x})\big)^{-},
\end{array}
\]
which amounts to saying that MOQ holds at $\widehat{x}.$\newline
\end{proof}

Obviously, the unique objective qualification in the above list, MOQ, is
independent of the remaining data qualifications. Notice that, under the SSCQ,
if $\psi$ is finite-valued, by \cite[Proposition 1.3.3]{HirLem},%
\begin{equation}%
\begin{array}
[c]{rr}%
\operatorname*{int}S & =\left\{  x\in\mathbb{R}^{n}:\psi\left(  x\right)
<0\right\}  ,\\
\operatorname*{bd}S & =\left\{  x\in\mathbb{R}^{n}:\psi\left(  x\right)
=0\right\}  ,\\
\overline{S} & =\left\{  x\in\mathbb{R}^{n}:\psi\left(  x\right)
\leq0\right\}  .
\end{array}
\label{2.2}%
\end{equation}
In the particular case that $\left(  P\right)  $ is continuous and the SCQ
holds then, by \cite[Theorem 7.9]{GorLop}, $\widehat{x}\in\operatorname*{bd}S$
if and only if $G_{\ast}(\widehat{x})\neq\left\{  0_{n}\right\}  ,$ which
entails $G(\widehat{x})\neq\emptyset.$

We illustrate the above data qualifications with a simple example showing that
the PLVCQ is not implied by the remaining data qualifications.

\begin{example}
\label{Example1}Consider the linear MOSIP problem in one variable%
\[%
\begin{array}
[c]{lll}%
\left(  P\right)  & \text{minimize} & f\left(  x\right)  =\left(
-2x,-x\right) \\
& \text{subject to } & g_{0}(x)=2x,\\
&  & g_{2k+1}(x)=x-\frac{1}{k+1},\text{ }k=0,1,2,\ldots,\\
&  & g_{2k}(x)=3x-\frac{1}{k},\text{ }k=1,2,\ldots.
\end{array}
\]
We observe that $-1$ is strong Slater point, $S=-\mathbb{R}_{+},$
$C(S,\widehat{x})=-\mathbb{R}_{+},$ $N\left(  S,\widehat{x}\right)
=\mathbb{R}_{+}, $ $T(\widehat{x})=\left\{  0\right\}  ,$ $G(\widehat
{x})=\left\{  2\right\}  ,$ $G^{0}(\widehat{x})=-\mathbb{R}_{+},$
$G^{-}(\widehat{x})=-\mathbb{R}_{++},$ $G_{\ast}(\widehat{x})=\mathbb{R}_{+},$
$F(\widehat{x})=\left[  -2,-1\right]  , $
\[
\psi(x)=\left\{
\begin{array}
[c]{ll}%
x, & \text{if}\ x<0,\\
3x, & \text{else;}%
\end{array}
\right.  \psi^{\prime}(\widehat{x};d)=\left\{
\begin{array}
[c]{ll}%
-1, & \text{if }d<0,\\
0, & \text{if }d=0,\\
3, & \text{else,}%
\end{array}
\right.
\]
and $\partial\psi(\widehat{x})=[1,3].$ Moreover, $F^{0}(\widehat{x})\cap
G^{0}(\widehat{x})=\mathbb{R}_{+}\cap\left(  -\mathbb{R}_{+}\right)  =\left\{
0\right\}  ,$ $C(Q^{1}\left(  \widehat{x}\right)  ,\widehat{x})\cap
C(Q^{2}\left(  \widehat{x}\right)  ,\widehat{x})=\left\{  0\right\}  ,$
$\big(\partial f_{1}(\widehat{x})\cup\partial f_{2}(\widehat{x})\big)^{-}%
=\left\{  -2,-1\right\}  ^{-}=\mathbb{R}_{++},$ and $\sup_{t\in T_{\varepsilon
}(\widehat{x})}g_{t}^{\prime}(\widehat{x})=\sup\left\{  1,2,3\right\}  =3$
independently of $\varepsilon>0.$ Thus, SCQ, SSCQ, PMFCQ, LFMCQ, COCQ, MFCQ,
KTCQ, CCCQ, ACQ, WADQ, EADQ, and MOQ hold at $\widehat{x}$ while PLVCQ fails
at that point.
\end{example}

\begin{theorem}
[Connections between data qualifications]\label{ThCQs}Let $\widehat{x}\in S.$
The implications of Diagram 1, where the local conditions are referred to
$\widehat{x}$ and the label $\left[  1\right]  $ ($\left[  2\right]  ,$
$\left[  3\right]  ,$ resp.) besides an arrow stands for `the implication
holds under the assumption that $(P)$ is continuous ($G\left(  \widehat
{x}\right)  \neq\emptyset,$ $p=1,$ resp.)', hold true.
\begin{align*}
& \bigskip%
\begin{array}
[c]{ccccccccc}%
\text{SSCQ} & \leftrightarrows^{\left[  1\right]  } & \text{SCQ} &  &
\text{COCQ} & \rightarrow^{\left[  1\right]  } & \text{KTCQ} & \leftarrow &
\left\{
\begin{array}
[c]{c}%
\text{KTCQ}\\
\text{and PLVCQ}%
\end{array}
\right\} \\
&  & \downarrow^{\left[  2\right]  } & \searrow^{\left[  1,2\right]  } &
\uparrow^{\left[  1\right]  } &  &  & ^{\left[  2\right]  }\swarrow & \\
\text{PMFCQ} & \rightarrow^{\left[  2\right]  } & \text{MFCQ} & \leftarrow &
\left\{
\begin{array}
[c]{c}%
\text{MFCQ}\\
\text{and PLVCQ}%
\end{array}
\right\}  & \rightarrow & \text{ACQ} & \rightarrow^{\left[  3\right]  } &
\text{EADQ}\\
&  & \downarrow^{\left[  1\right]  } &  &  & \nearrow & \downarrow & \swarrow
& \\
&  & \text{LFMCQ} & \leftrightarrow & \left\{
\begin{array}
[c]{c}%
\text{ACQ}\ \\
\text{and CCCQ}%
\end{array}
\right\}  &  & \text{WADQ} &  &
\end{array}
\\
& \text{
\ \ \ \ \ \ \ \ \ \ \ \ \ \ \ \ \ \ \ \ \ \ \ \ \ \ \ \ \ \ \ \ \ \ \ \ \ \ }%
\begin{array}
[c]{c}%
\\
\text{Diagram 1}%
\end{array}
\end{align*}

\end{theorem}

\begin{proof}
We first examine the case when $G(\widehat{x})=\emptyset$ regarding those
implications involving the qualification conditions related with
$G^{-}(\widehat{x})$ and $G^{0}(\widehat{x}),$ that is, MFCQ, ACQ, WADQ, and
EADQ. When [2] (i.e., the negation of\ $G(\widehat{x})=\emptyset$)\ is
required, there is nothing to be proved. On the one hand, since MFCQ fails,
[MFCQ $\Longrightarrow$ LFMCQ] holds. On the other hand, since MFCQ, PLVCQ,
ACQ, WADQ, and EADQ fail simultaneously, all implications of Diagram 1
involving ACQ, WADQ, and EADQ also hold. So, we can assume without loss of
generality that $G(\widehat{x})\neq\emptyset.$\newline[SCQ $\Longrightarrow$
SSCQ]:\ It is a straightforward consequence of the continuity assumption on
$\left(  P\right)  $.\newline[PMFCQ $\Longrightarrow$ MFCQ]: Let
$\varepsilon>0$ be such that $\sup\left\{  \xi^{\prime}x_{\ast}\mid\xi\in%
%TCIMACRO{\dbigcup \limits_{t\in T_{\varepsilon}(\widehat{x})}}%
%BeginExpansion
{\displaystyle\bigcup\limits_{t\in T_{\varepsilon}(\widehat{x})}}
%EndExpansion
\partial g_{t}(\widehat{x})\right\}  <0.$ Since $G\left(  \widehat{x}\right)
\subset%
%TCIMACRO{\dbigcup \limits_{t\in T_{\varepsilon}(\widehat{x})}}%
%BeginExpansion
{\displaystyle\bigcup\limits_{t\in T_{\varepsilon}(\widehat{x})}}
%EndExpansion
\partial g_{t}(\widehat{x}),$ $\xi^{\prime}x_{\ast}<0$ for all $\xi\in
G\left(  \widehat{x}\right)  ,$ i.e., $x_{\ast}\in G^{-}(\widehat{x}%
).$\newline[SCQ $\Longrightarrow$ MFCQ]: If $x_{0}$ is a Slater point, $t\in
T(\widehat{x})$ and $\xi_{t}\in\partial g_{t}(\widehat{x})$, then $\xi
_{t}^{\prime}(x_{0}-\widehat{x})\leq g_{t}(x_{0})-g_{t}(\widehat{x})<0$, so
that MFCQ holds at $\widehat{x}$ (actually at any point of $S$).\newline[SCQ
$\Longrightarrow$PLVCQ]: It is \cite[Pshenichnyi-Levin-Valadier theorem, p.
267]{HirLem} (the continuity assumption is essential).\newline[MFCQ
$\Longrightarrow$ LFMCQ]: It is \cite[Theorem 14(iii)]{GOVATO2}.\newline[LFMCQ
$\Longrightarrow$ ACQ]: By the assumption and the closedness of $C(S,\widehat
{x}),$
\[
G^{0}(\widehat{x})=G_{\ast}^{0}(\widehat{x})=N^{0}(S,\widehat{x}%
)=C^{00}(S,\overline{x})=\overline{C(S,\widehat{x})}=C(S,\widehat{x}).
\]
\newline[LFMCQ $\Longrightarrow$ CCCQ]: It follows from the closedness of
$N(S,\widehat{x}).$\newline[(ACQ$\ \wedge$ CCCQ) $\Longrightarrow$ LFMCQ]:
Taking negative polar cones in both sides of the inclusion $G_{\ast}%
^{0}(\widehat{x})=G^{0}(\widehat{x})\subseteq C(S,\widehat{x}),$ and recalling
the closedness of $G_{\ast}(\widehat{x}),$ one has
\[
N(S,\widehat{x})=C^{0}(S,\widehat{x})\subseteq G_{\ast}^{00}(\widehat
{x})=G_{\ast}(\widehat{x}),
\]
while the reverse inclusion holds by (\ref{2.1}).\newline[(KTCQ$\ \wedge$
PLVCQ) $\Longrightarrow$ ACQ]: According to \cite[Theorem 23.2]{ROC},%
\[%
\begin{array}
[c]{ll}%
\big(\partial\psi(\widehat{x})\big)^{0} & =\{d\in\mathbb{R}^{n}\mid
\psi^{\prime}(\widehat{x};d)\geq\xi^{\prime}d\Longrightarrow\xi^{\prime}%
d\leq0\}\\
& =\{d\in\mathbb{R}^{n}\mid\psi^{\prime}(\widehat{x};d)\leq0\},
\end{array}
\]
which, together with $G(\widehat{x})\neq\emptyset,$ PLVCQ and KTCQ, yields
\[
G^{0}(\widehat{x})=\left(  \operatorname*{conv}\left(  G(\widehat{x})\right)
\right)  ^{0}\subseteq\big(\partial\psi(\widehat{x})\big)^{0}=\{d\in
\mathbb{R}^{n}\mid\psi^{\prime}(\widehat{x};d)\leq0\}\subseteq C(S,\widehat
{x}).
\]
[(MFCQ $\wedge$ PLVCQ) $\Longrightarrow$ COCQ]: Let $d\in G^{-}(\widehat{x})$.
Since $G^{-}(\widehat{x})=\Big(\operatorname*{conv}\big(G(\widehat
{x})\big)\Big)^{-},$ the PLVCQ leads to $d\in\big(\partial\psi(\widehat
{x})\big)^{-}.$ The continuity assumption [1] guarantees that $\widehat{x}%
\in\operatorname*{int}\operatorname{dom}\psi=\mathbb{R}^{n},$ so that
$\partial\psi\left(  \widehat{x}\right)  $ is compact.\ Hence, by
\cite[Theorem 23.4]{ROC},
\[
\psi^{\prime}(\widehat{x};d)=\max\left\{  u^{\prime}d\mid u\in\partial
\psi\left(  \widehat{x}\right)  \right\}  <0.
\]
\newline[(MFCQ $\wedge$ PLVCQ) $\Longrightarrow$ ACQ]: By the latter proof, we
can take $d\in G^{-}(\widehat{x})$ such that $\psi^{\prime}(\widehat{x};d)<0.$
Then, there exists a scalar $\delta>0$ such that $\psi(\widehat{x}+\beta
d)<\psi(\widehat{x})\leq0,$ for all $\beta\in(0,\delta].$ Therefore, we have
$\widehat{x}+\beta d\in S$ for all $\beta\in(0,\delta]$ $,$ which implies
$d\in D(S,\widehat{x}).$ We have thus proved the inclusion $G^{-}(\widehat
{x})\subseteq D(S,\widehat{x}).$ Hence, we get
\[
G^{0}(\widehat{x})=\overline{G^{-}(\widehat{x})}\subseteq\overline
{D(S,\widehat{x})}=C(S,\widehat{x}).
\]
\newline[COCQ $\Longrightarrow$ KTCQ]: Since $\psi^{\prime}(\widehat
{x};\widehat{d})<0$, by the same argument as in the proof of [(MFCQ $\wedge$
PLVCQ) $\Longrightarrow$ ACQ], we obtain a $\delta>0$ such that $\widehat
{x}+\delta\widehat{d}\in S.$ Thus, recalling the definition of
subdifferential, for each $t\in T(\widehat{x})$ and $\xi\in\partial
g_{t}(\widehat{x})$, we have
\[
\xi^{\prime}(\delta\widehat{d})=\xi^{\prime}(\widehat{x}+\delta\widehat
{d}-\widehat{x})\leq g_{t}(\widehat{x}+\delta\widehat{d})-g_{t}(\widehat
{x})=g_{t}(\widehat{x}+\delta\widehat{d})\leq0.
\]
This means that $\{d\in\mathbb{R}^{n}\mid\psi^{\prime}(\widehat{x}%
;d)<0\}\subseteq G^{0}(\widehat{x})$. This, together with the continuity of
$\psi^{\prime}(\widehat{x};.),$ implies that
\[
\{d\in\mathbb{R}^{n}\mid\psi^{\prime}(\widehat{x};d)\leq0\}=\overline
{\{d\in\mathbb{R}^{n}\mid\psi^{\prime}(\widehat{x};d)<0\}}\subseteq
\overline{G^{0}(\widehat{x})}=G^{0}(\widehat{x}).
\]
The conclusion is immediate from the last inclusion and the definition of
KTCQ.\newline[ACQ $\Longrightarrow$ EADQ] When $p=1$,\ the ACQ implies that%
\[
F^{0}(\widehat{x})\cap G^{0}(\widehat{x})\subseteq G^{0}(\widehat{x})\subseteq
C\left(  S,\widehat{x}\right)  =C(Q^{1}(\widehat{x}),\widehat{x}).
\]
[EADQ $\Longrightarrow$ WADQ] In fact,%
\[
F^{-}(\widehat{x})\cap G^{0}(\widehat{x})\subseteq F^{0}(\widehat{x})\cap
G^{0}(\widehat{x})\subseteq\bigcap_{i=1}^{p}C(Q^{i}(\widehat{x}),\widehat
{x})\subseteq C\left(  S,\widehat{x}\right)  .
\]
The proof is complete.\newline
\end{proof}

\begin{remark}
Notice that assumption [1] can be relaxed in some cases. For instance, in the
proofs of [(MFCQ $\wedge$ PLVCQ) $\Longrightarrow$ COCQ], [(MFCQ $\wedge$
PLVCQ) $\Longrightarrow$ ACQ], and [COCQ $\Longrightarrow$ KTCQ]\ we have used
the compactness of $\partial\psi\left(  \widehat{x}\right)  ,$ which follows
from [1], but also from $\widehat{x}\in\operatorname*{int}S$ (as
$S\subset\operatorname{dom}\psi$).
\end{remark}

The next , where we do not specify the objective function $f,$ shows the
necessity of the additional hypothesis [2] in the four implications of Diagram
1 where it is assumed.

\begin{example}
\label{Example3}Take $n=1,T=\left[  1,2\right]  ,$ $\widehat{x}=0,$ and%
\[
g_{t}(x)=\left\{
\begin{array}
[c]{ll}%
-\sqrt{2tx-x^{2}}, & \text{if }x\in\left[  0,2t\right]  ,\\
+\infty,\, & \text{\ otherwise.}%
\end{array}
\right.
\]
The function $\left(  t,x\right)  \longmapsto g_{t}(x)$ is continuous on
$T\times S=\left[  1,2\right]  \times\left[  0,2\right]  \,\ $but not on
$T\times\mathbb{R}^{n},$ so that [1] fails. One has $T(\widehat{x})=T,$
$\psi=g_{1},$ $C\left(  S,\widehat{x}\right)  =\mathbb{R}_{+},$ $\psi^{\prime
}(\widehat{x};d)=-\infty$ for all $d>0,$ $\partial g_{t}(\widehat
{x})=\emptyset$ for all $t\in T,\ G(\widehat{x})=\emptyset$ (i.e., [2] fails
despite of $\widehat{x}\in\operatorname*{bd}S$), and $G_{\ast}(\widehat
{x})=\left\{  0\right\}  .$ Moreover, $1$ is strong Slater point (with slack
$1$) and, taking an arbitrary $x_{\ast}\in\mathbb{R}^{n},$ we get $\left\{
\xi^{\prime}x_{\ast}\mid\xi\in%
%TCIMACRO{\dbigcup \limits_{t\in T_{\varepsilon}(\widehat{x})}}%
%BeginExpansion
{\displaystyle\bigcup\limits_{t\in T_{\varepsilon}(\widehat{x})}}
%EndExpansion
\partial g_{t}(\widehat{x})\right\}  =\emptyset$ for all $\varepsilon>0,$ so
that
\[
\inf_{\varepsilon>0}\sup\left\{  \xi^{\prime}x_{\ast}\mid\xi\in%
%TCIMACRO{\dbigcup \limits_{t\in T_{\varepsilon}(\widehat{x})}}%
%BeginExpansion
{\displaystyle\bigcup\limits_{t\in T_{\varepsilon}(\widehat{x})}}
%EndExpansion
\partial g_{t}(\widehat{x})\right\}  =-\infty<0.
\]
Thus, SSCQ, SCQ, COCQ, KTCQ, PMFCQ, and PLVCQ hold while MFCQ, LFMCQ and ACQ
fail. Then, [PMFCQ$\Longrightarrow$ MFCQ], [SCQ $\Longrightarrow$ MFCQ], [SCQ
$\Longrightarrow$ (MFCQ $\wedge$ PLVCQ)], and [(KTCQ$\ \wedge$ PLVCQ)
$\Longrightarrow$ ACQ] fail.
\end{example}

Theorem \ref{ThCQs}\ shows roughly speaking, first, that CCCQ and PLVCQ play
subsidiary roles with respect to other CQs, second, that SSCQ is the strongest
data qualification among those that are included in Diagram 1 and, third, that
LFMCQ, KTCQ and WADQ are the weakest.

\begin{example}
\label{Example2}Consider the MOSIP problem in $\mathbb{R}^{2}$
\[%
\begin{array}
[c]{lll}%
\left(  P\right)  & \text{minimize} & f\left(  x\right)  =\left(
-x_{1},-x_{1}\right) \\
& \text{subject to } & g_{t}(x)\leq0,t\in T:=\mathbb{N}\cup\{0\},
\end{array}
\]
where $f_{1}(x)=f_{2}(x):=-x_{1},$ and $g_{t}(x):=\sup\big\{x^{\prime}y\mid
y\in X_{t}\big\}$ is the \emph{support function} of the compact convex set
$X_{t}:=\big\{x\in\mathbb{R}_{+}^{2}\mid x_{1}^{2}+x_{2}^{2}-2(1+t)x_{2}%
\leq0\},$ so that $g_{t}\in\Gamma\left(  \mathbb{R}^{n}\right)  $ for all
$t\in T.$ Since $\left\{  X_{t}\right\}  _{t=1}^{\infty}$ is an expansive
sequence of sets, $\left\{  g_{t}\right\}  _{t=1}^{\infty}$\ is a
non-decreasing sequence of nonnegative functions. Moreover, given $t\in
T,$\ by a well-known property of the support functions, $\partial
g_{t}(\widehat{x})=X_{t},$ while elementary calculus yields the explicit
expression%
\[
g_{t}\left(  x\right)  =\left\{
\begin{array}
[c]{ll}%
2(1+t)\max\left\{  x_{2},0\right\}  , & \text{if }x_{1}\leq0,\\
\left\Vert x\right\Vert +(1+t)x_{2}, & \text{else,}%
\end{array}
\right.
\]
for the constraint functions. Thus, $\psi\left(  x\right)  =0$ when
$x\in-\mathbb{R}_{+}^{2}$ and $\psi\left(  x\right)  =+\infty$ otherwise, so
that $S=-\mathbb{R}_{+}^{2}.$\newline We now take $\widehat{x}=0_{2}.$ Then,
$\psi^{\prime}\left(  \widehat{x};d\right)  =0,$ if $d\in-\mathbb{R}_{+}^{2},$
and $\psi^{\prime}\left(  \widehat{x};d\right)  =+\infty,$ otherwise while
$\partial\psi\left(  \widehat{x}\right)  =\mathbb{R}_{+}^{2}.$
Moreover,\ $C(S,\widehat{x})=-\mathbb{R}_{+}^{2},$ $N\left(  S,\widehat
{x}\right)  =\mathbb{R}_{+}^{2},$ $T(\widehat{x})=T,$
\[
G_{\ast}(\widehat{x})=G(\widehat{x})=\big\{x\in\mathbb{R}^{2}\mid x_{1}%
\geq0,\ x_{2}>0\big\}\cup\{0_{2}\},
\]
$G^{0}(\widehat{x})=-\mathbb{R}_{+}^{2},$ $G^{-}(\widehat{x})=\emptyset,$
$F(\widehat{x})=\{(-1,0)\},$ $F^{0}(\widehat{x})\cap G^{0}(\widehat
{x})=\left\{  0\right\}  \times\left(  -\mathbb{R}_{+}\right)  ,$
$C(Q^{1}(\widehat{x}),\widehat{x})\cap C(Q^{2}(\widehat{x}),\widehat
{x})=\left\{  0\right\}  \times\left(  -\mathbb{R}_{+}\right)  ,$
$\bigcup_{i=i}^{p}\big(\partial f_{i}(\widehat{x})\big)^{-}=\mathbb{R}%
_{++}\times\mathbb{R},$ and $%
%TCIMACRO{\dbigcup \limits_{t\in T_{\varepsilon}(\widehat{x})}}%
%BeginExpansion
{\displaystyle\bigcup\limits_{t\in T_{\varepsilon}(\widehat{x})}}
%EndExpansion
\partial g_{t}(\widehat{x})=\left(  \mathbb{R\times R}_{++}\right)
\cup\left\{  0_{2}\right\}  $ for all $\varepsilon>0.$ Therefore, KTCQ, ACQ,
WADQ, and EADQ hold at $\widehat{x}$ while the remaining data qualifications
in Theorem \ref{ThCQs} fail (as well as \textit{MOQ)}. Observe also that
$\operatorname*{int}S\neq\emptyset=\left\{  x\in\mathbb{R}^{n}:\psi\left(
x\right)  <0\right\}  ,$ so that SSCQ is essential for the validity of
(\ref{2.2}).
\end{example}

\section{Optimality conditions for weak efficiency}

Recall that we consider a problem $(P)$ as in (\ref{P}) with convex and
finite-valued objective functions $f_{i},$ $i\in I,$ and constraint functions
$g_{t}\in\Gamma\left(  \mathbb{R}^{n}\right)  ,$ $t\in T.$

We say that the \emph{weak KKT condition} holds at $\widehat{x}\in S$ when
there exist $\alpha_{i}\geq0$ for $i\in I$ with $\sum_{i=1}^{p}\alpha_{i}=1$,
and $\beta_{t}\geq0$ for $t\in T(\widehat{x})$, with $\beta_{t}\neq0$ for
finitely many indexes, such that
\begin{equation}
0_{n}\in\sum_{i=1}^{p}\alpha_{i}\partial f_{i}(\widehat{x})+\sum_{t\in
T(\widehat{x})}\beta_{t}\partial g_{t}(\widehat{x}).\label{***}%
\end{equation}
In geometric terms, the weak KKT condition holds at $\widehat{x}\in S$ if and
only if $0_{n}\in F_{\ast}(\widehat{x})+G_{\ast}(\widehat{x}).$

\begin{theorem}
[Weak KKT necessary condition under WACQ and CCCQ]\label{3.2}Let $\widehat{x}$
be a weak efficient solution of problem $(P).$ Then:\newline$(i)$ If WACQ
holds at $\widehat{x},$ one has
\begin{equation}
0_{n}\in F_{\ast}(\widehat{x})+\overline{G_{\ast}(\widehat{x})}.\label{**}%
\end{equation}
\newline$(ii)$ If, in addition, CCCQ holds at $\widehat{x},$ then $(P)$
satisfies the weak KKT condition at $\widehat{x}.$
\end{theorem}

\begin{proof}
$(i)$ We first claim that
\begin{equation}
\max_{i\in I}f_{i}^{\prime}(\widehat{x};d)\geq0,\ \ \forall d\in
D(S,\widehat{x}).\label{111}%
\end{equation}
On the contrary, suppose that there exists a $d\in D(S,\widehat{x})$ such that
$f_{i}^{\prime}(\widehat{x};d)<0$ for all $i\in I$. Thus, there exist positive
scalars $\delta,\delta_{1},,,,,\delta_{p}$ such that
\begin{equation}
\left\{
\begin{array}
[c]{ll}%
\widehat{x}+\varepsilon d\in S, & \forall\varepsilon\in(0,\delta),\\
f_{i}(\widehat{x}+\varepsilon d)-f_{i}(\widehat{x})<0, & \forall\varepsilon
\in(0,\delta_{i}).
\end{array}
\right. \label{3.1}%
\end{equation}
Take $\widehat{\delta}:=\min\{\delta,\delta_{1},,,,,\delta_{p}\}$. From
(\ref{3.1}), for each $\varepsilon\in(0,\widehat{\delta})$ we get
$f(\widehat{x}+\varepsilon d)<f(\widehat{x})$ and $\widehat{x}+\varepsilon
d\in S,$ which contradicts the weak efficiency of $\widehat{x}$. Thus,
(\ref{111}) is true.\newline\newline We now show that (\ref{111}) also holds
for $d\in C(S,\widehat{x})=\overline{D(S,\widehat{x})}.$ Indeed, if
$d\in\overline{D(S,\widehat{x})}$, there exists a sequence $\{d_{k}%
\}_{k=1}^{\infty}$ in $D(S,\widehat{x})$ converging to $d$. For each $i\in I,$
since $f_{i}$ is a finite convex function, its directional derivative function
$f_{i}^{\prime}(\widehat{x};\cdot)$\ at $\widehat{x}$ is finite sublinear
\cite[Proposition 1.1.2]{HirLem} and, so, convex and continuous. Thus,
$\varphi(\cdot):=\max_{i\in I}f_{i}^{\prime}(\widehat{x};\cdot)$ is a convex
continuous function too. From (\ref{111}) and the continuity of $\varphi$ we
deduce that $\varphi(d)=\lim_{k\rightarrow\infty}\varphi(d_{k})\geq0,$ so that
$\varphi(d)\geq0$ for all $d\in C\left(  S,\widehat{x}\right)  $. From this
inequality and the ACQ at $\widehat{x}$, we obtain that
\begin{equation}
\varphi(d)\geq0,\text{ }\forall d\in F^{-}(\hat{x})\cap G^{0}(\widehat
{x}).\label{3.3}%
\end{equation}
\newline\newline We claim that $F^{-}(\hat{x})\cap G^{0}(\hat{x})=\emptyset$.
Otherwise, if $\hat{d}\in F^{-}(\hat{x})\cap G^{0}(\hat{x})$, then
$f_{i}^{\prime}(\hat{x};\hat{d})<0$ for all $i\in I$ (by definition of
$F^{-}(\hat{x})$), so that $\varphi(\hat{d})<0$, which contradicts
(\ref{3.3}). Thus our claim is proved. Since $F^{-}(\hat{x})=F_{\ast}^{-}%
(\hat{x})$ and $G^{0}(\hat{x})=\big(\overline{G_{\ast}(\hat{x})}\big)^{0}$,
then $F_{\ast}^{-}(\hat{x})\cap\big(\overline{G_{\ast}(\hat{x})}%
\big)^{0}=\emptyset$. Hence there is no vector $v\in\mathbb{R}^{n}$
satisfying
\[
\left\{
\begin{array}
[c]{ll}%
v^{\prime}y<0,\ \forall y\in F_{\ast}(\hat{x}), & \\
v^{\prime}y\geq0,\ \ \forall y\in\big(-\overline{G_{\ast}(\hat{x})}\big). &
\end{array}
\right.
\]
Since $F_{\ast}(\hat{x})$ is a non-empty convex set and $\big(-\overline
{G_{\ast}(\hat{x})}\big)$ is a closed convex cone, by the strong separation
theorem (see, e.g., \cite[Corollary 11.4.1]{ROC} we get $F_{\ast}(\hat{x}%
)\cap\big(-\overline{G_{\ast}(\hat{x})}\big)\neq\emptyset.$ This means that
\[
0_{n}\in F_{\ast}(\hat{x})\cap\overline{G_{\ast}(\hat{x})}\big).
\]
\newline$(ii)$ Under the closedness assumption, $0_{n}\in F_{\ast}(\widehat
{x})+G_{\ast}(\widehat{x})$ and the conclusion follows. \newline
\end{proof}

\begin{theorem}
[Weak KKT sufficient condition]\label{qwe}\textbf{\ } If the weak KKT
condition holds at $\widehat{x}\in S,$ then, $\widehat{x}$ is a weak efficient
solution of $(P)$.
\end{theorem}

\begin{proof}
Let $\alpha_{i}\geq0,$ $i\in I,$ with $\sum_{i=1}^{p}\alpha_{i}=1$, and some
$\beta_{t}\geq0,$ $t\in T(\widehat{x})$, with $\beta_{t}\neq0$ for finitely
many indexes.Due to (\ref{***}), we can find some $\xi_{i}\in\partial
f_{i}(\widehat{x})$ and $\zeta_{t}\in\partial g_{t}(\widehat{x})$ for
$(i,t)\in I\times T(\widehat{x})$ such that
\begin{equation}
\sum_{i=1}^{p}\alpha_{i}\xi_{i}+\sum_{t\in T^{\ast}}\beta{_{t}}\zeta_{t}%
=0_{n},\label{bbb1}%
\end{equation}
where $T^{\ast}:=\{t\in T(\widehat{x})\mid\beta{_{t}}\neq0\}$. Suppose on the
contrary that $\widehat{x}$ is not a weak efficient solution for $(P)$. Then
there exists a feasible point $x_{\ast}$ for $(P)$ such that $f_{i}(x_{\ast
})<f_{i}(\widehat{x})$ for all $i\in I$. Thus, $\xi_{i}^{\prime}(x_{\ast
}-\widehat{x})<0$ for all $i\in I$. Due to the last inequality, it follows
from (\ref{bbb1}) that
\begin{equation}
\sum_{t\in T^{\ast}}\beta{_{t}}\zeta_{t}^{\prime}(x_{\ast}-\widehat{x}%
)=-\sum_{i=1}^{p}\alpha_{i}\xi_{i}^{\prime}(x_{\ast}-\widehat{x})>0.\label{b1}%
\end{equation}
On the other hand, since $T^{\ast}\subseteq T(\widehat{x})$, we obtain that
\[
\sum_{t\in T^{\ast}}\beta{_{t}}\zeta_{t}^{\prime}(x_{\ast}-\widehat{x})\leq0.
\]
This contradicts (\ref{b1}).
\end{proof}

Combining Theorems \ref{ThCQs}, \ref{3.2} and \ref{qwe} we get the following
characterization of the weak efficient solution for (possibly non-continuous)
convex MOSIP problems.

\begin{corollary}
[Characterization under LFMCQ via weak KKT condition]\label{k4}Suppose that
$(P)$ satisfies LFMCQ at $\widehat{x}\in S.$ Then, $\widehat{x}$\ is a weak
efficient solution for $(P)$ if and only if $(P)$ satisfies the weak KKT
condition at $\widehat{x}.$
\end{corollary}

A similar result was proved, for continuous convex MOSIP problems, in
\cite[Theorem 27]{GOVATO2} under the MFCQ and in \cite[Theorem 29]{GOVATO2}
under the KTCQ and the CCCQ, and the assumption that $0_{n}\notin F_{\ast
}\left(  \widehat{x}\right)  .$

We now exploit the gap function associated with $(P)$ to characterize its weak
efficient solutions.

\begin{theorem}
[Characterization under LFMCQ via gap function]\label{gap1}Let $\widehat{x}\in
S.$ The following statements hold: \newline$(i)$ If $\vartheta(\widehat{x}%
,\xi,\lambda)=0$ for some $\xi:=(\xi_{1},\ldots,\xi_{p})\in\prod_{i=1}%
^{p}\Big(\partial f_{i}\left(  \widehat{x}\right)  \Big)$ and $\lambda
\geq0_{p}$, then $\widehat{x}$ is a weak efficient solution for $(P).$%
\newline$(ii)$ If $\widehat{x}$ is a weak efficient solution for $(P)$ where
the LFMCQ holds, then there exist $\xi\in\prod_{i=1}^{p}\partial
f_{i}(\widehat{x})$ and $\lambda\geq0_{p}$ such that $\vartheta(\widehat
{x},\xi,\lambda)=0.$\newline
\end{theorem}

\begin{proof}
$(i)$ Assume that $\vartheta(\widehat{x},\xi,\lambda)=0$, while $\widehat{x}$
is not a weak efficient solution for $(P)$. Then, there exists $x^{\ast}\in S$
such that
\[
f_{i}(x^{\ast})<f_{i}(\widehat{x}),\ \forall i\in I.\
\]
Then, by definition of subgradient,%
\[
\xi_{i}^{\prime}(x^{\ast}-\widehat{x})<0,\ \forall i\in I.
\]
Due to the latter inequalities and the assumption that $\lambda\geq0_{p}$, we
have $\sum_{i=1}^{p}\lambda_{i}\xi_{i}^{\prime}(\widehat{x}-x^{\ast})>0.$
Hence, $\vartheta(\widehat{x},\xi,\lambda)>0$, which contradicts the
assumption.\newline\newline$(ii)$ According to Corollary \ref{k4}, there exist
$\lambda:=\left(  \lambda_{1},,,,,\lambda_{p}\right)  \geq0_{p}$ with
$\sum_{i=1}^{p}\lambda_{i}=1$, a finite set $\left\{  t_{1},,,,,t_{q}\right\}
\subseteq T\left(  \widehat{x}\right)  ,$ with corresponding nonnegative
scalars $\mu_{t_{1}},,,,,\mu_{t_{q}},$ and subgradients $\xi_{i}\in\partial
f_{i}\left(  \widehat{x}\right)  $ for $i\in I$, and $\zeta_{t_{m}}\in\partial
g_{t_{m}}\left(  \widehat{x}\right)  $ for $m=1,,,,,q$, such that
\begin{equation}
\sum_{i=1}^{p}\lambda_{i}\xi_{i}+\sum_{m=1}^{q}\mu_{t_{m}}\zeta_{t_{m}}%
=0_{n}\label{330}%
\end{equation}
Take an arbitrary $y\in S$ and $m\in\left\{  1,...,q\right\}  .$ Since
$g_{t_{m}}\left(  y\right)  \leq0=g_{t_{m}}\left(  \widehat{x}\right)  ,$
$\zeta_{t_{m}}^{\prime}(y-\widehat{x})\leq0.$ This and (\ref{330}) imply
that:
\[
\sum_{i=1}^{p}\lambda_{i}\xi_{i}^{\prime}(y-\widehat{x})=-\sum_{m=1}^{q}%
\mu_{t_{m}}\zeta_{t_{m}}^{\prime}(y-\widehat{x})\geq0.
\]
Therefore, $\sum_{i=1}^{p}\lambda_{i}\xi_{i}^{\prime}(\widehat{x}-y)\leq0 $
for all $y\in S.$ From this and $\sum_{i=1}^{p}\lambda_{i}\xi_{i}^{\prime
}(\widehat{x}-\widehat{x})=0$, we conclude that
\[
\vartheta(x,\xi,\lambda)=\sup_{y\in S}\Big\{\sum_{i=1}^{p}\lambda_{i}\xi
_{i}^{\prime}(x-y)\Big\}=0.\newline%
\]

\end{proof}

Finally in this section, we consider $(P)$ as a particular instance of the
general vector optimization problem studied in \cite{DIGOLOLO}, just taking as
decision space $X=\mathbb{R}^{n},$ as objective space $Y=\mathbb{R}^{p},$ and
as constraint space $Z$ some linear subspace of $\mathbb{R}^{T}$ such that
$Z^{\bullet}:=Z\cup\{+\infty_{Z}\},$ where $+\infty_{Z}$ denotes a greatest
element, contains $\left\{  g\left(  x\right)  \mid x\in\mathbb{R}%
^{n}\right\}  ,$ where $g\left(  x\right)  :=\left(  g_{t}\left(  x\right)
\right)  _{t\in T}$ for all $x\in\mathbb{R}^{n}$, with the convention that
$g\left(  x\right)  =+\infty_{Z}$ when $g_{t}\left(  x\right)  =+\infty$ for
at least one $t\in T.$ The \emph{domain} of $g$ is
\[
\operatorname{dom}g:=\{x\in\mathbb{R}^{n}\mid g\left(  x\right)  \neq
+\infty_{Z}\}=%
%TCIMACRO{\dbigcap \limits_{t\in T}}%
%BeginExpansion
{\displaystyle\bigcap\limits_{t\in T}}
%EndExpansion
\operatorname{dom}g_{t}\neq\emptyset.
\]
We assume that $Z$ is equipped with a locally convex topology finer than the
one of the pointwise convergence, with positive cone $Z\cap\mathbb{R}_{+}%
^{T}.$ We consider on the dual space of $Z,$ $Z^{\ast},$ the ordering induced
by the positive cone $Z\cap\mathbb{R}_{+}^{T},$ i.e., the positive cone in
$Z^{\ast}$ is $Z_{+}^{\ast}:=\{z^{\ast}\in Z^{\ast}\mid\left\langle z^{\ast
},z\right\rangle \geq0,\forall z\in Z\cap\mathbb{R}_{+}^{T}\}.$ Given
$z^{\ast}\in Z^{\ast},$\ we define $\left(  z^{\ast}\circ g\right)  \left(
x\right)  =+\infty$ whenever $g\left(  x\right)  =+\infty_{Z}.$

The \emph{conjugate} of a vector function $h:\mathbb{R}^{n}\longrightarrow
\mathbb{R}^{p}\cup\{+\infty_{\mathbb{R}^{p}}\},$ where $+\infty_{\mathbb{R}%
^{p}}$ denotes an element greater than any other in $\mathbb{R}^{p},$\ is the
set-valued map $h^{\ast}\colon\mathcal{M}_{p\times n}\rightrightarrows
\mathbb{R}^{p}\cup\{+\infty_{\mathbb{R}^{p}}\}$ defined by
\[
h^{\ast}(M):=\operatorname*{WSup}\{Mx-h(x)\mid x\in\mathbb{R}^{n}\},\text{
}\forall M\in\mathcal{M}_{p\times n},
\]
where $\operatorname*{WSup}V\ $represents the weak supremum of $V\subset
\mathbb{R}^{p}$ in Tanino's sense \cite{Tan92}. Due to the assumptions on the
data of $\left(  P\right)  ,$\ the next lemma is a straightforward consequence
of \cite[Theorem 5.1]{DIGOLOLO}.

\begin{lemma}
\label{OCA1}Let $\widehat{x}\in S.$ Then the following statements are
equivalent:\newline$(i)$ The set
\begin{equation}
\bigcup\limits_{\left(  z_{1}^{\ast},...,z_{p}^{\ast}\right)  \in\left(
Z_{+}^{\ast}\right)  ^{p}}\left\{  (M,y)\in\mathcal{M}_{p\times n}%
\times\mathbb{R}^{p}\mid y\in(f+\left(  z_{1}^{\ast},...,z_{p}^{\ast}\right)
\circ g)^{\ast}(M)+\mathbb{R}_{+}^{p}\right\} \label{closed set}%
\end{equation}
is closed regarding $(0_{p\times n},-f(\widehat{x}))$.\newline$\left(
ii\right)  $ $\widehat{x}$ is a weak efficient solution of $\left(  P\right)
$ if and only if there exist $z_{1}^{\ast},...,z_{p}^{\ast}\in Z_{+}^{\ast}$
such that
\begin{equation}
f(x)+\left(  \left\langle z_{1}^{\ast},g(x)\right\rangle ,...,\left\langle
z_{p}^{\ast},g(x)\right\rangle \right)  -f(\widehat{x})\notin-\mathbb{R}%
_{++}^{p},\;\forall x\in\mathbb{R}^{n}.\label{3.4}%
\end{equation}

\end{lemma}

Observe that (\ref{3.4}) holds whenever $\widehat{x}$ is a weak efficient
solution for the ordinary multiobjective problem%
\[
\ \text{minimize}\ f(x)\ \text{subject to}\ \left\langle z_{i}^{\ast
},g(x)\right\rangle \leq0,i\in I.
\]
We now get from Lemma \ref{OCA1}\ characterizations of the weak efficient
solutions involving continuous linear functionals on $Z$ instead of
subgradients of the data at $\widehat{x}$ (as in the KKT-type theorems) for
three particular types of convex MOSIP problems.

\begin{theorem}
[Characterization under SSCQ and SCQ via linear functionals]\label{ThZ}Let
$\widehat{x}\in S.$ Then the following statements hold:\newline$(i)$ Assume
that $T$ is countable and $0_{\mathbb{R}^{T}}\in\operatorname*{sqri}E,$ where
\begin{equation}
E:=g\left(
%TCIMACRO{\dbigcap \limits_{t\in T}}%
%BeginExpansion
{\displaystyle\bigcap\limits_{t\in T}}
%EndExpansion
\operatorname{dom}g_{t}\right)  +\mathbb{R}_{+}^{T}.\label{Set E}%
\end{equation}
Then $\widehat{x}$ is a weak efficient solution of $\left(  P\right)  $ if and
only if there exist multiplier vectors $\lambda^{1},...,$ $\lambda^{p}%
\in\mathbb{R}_{+}^{T},$ all of them with finite support, such that there is no
$x\in\mathbb{R}^{n}$ satisfying
\[
f_{i}(x)+%
%TCIMACRO{\dsum \nolimits_{t\in T}}%
%BeginExpansion
{\displaystyle\sum\nolimits_{t\in T}}
%EndExpansion
\lambda_{t}^{i}g_{t}(x)<f(\widehat{x}),\text{ }\forall i\in I.
\]
\newline$(ii)$ Assume that $T$ is a normal topological space space, the
functions $\left(  \cdot,x\right)  \mapsto g_{\cdot}\left(  x\right)  $ are
continuous for all $x\in\mathbb{R}^{n},$\ the supremum and infimum functions
$\psi$ and $\iota$\ (defined in (\ref{2.3}) and (\ref{2.4}), respectively) are
real-valued, and SSCQ holds. Then $\widehat{x}$ is a weak efficient solution
of $\left(  P\right)  $ if and only if there exist nonnegative bounded
finitely additive measures $\mu_{1},...,\mu_{p}$ such that there is no
$x\in\mathbb{R}^{n}$ satisfying
\[
f_{i}(x)+\int\nolimits_{T}g_{t}(x)d\mu_{i}\left(  t\right)  <f(\widehat
{x}),\;\forall i\in I.
\]
\newline$(iii)$ Assume that $\left(  P\right)  $ is continuous and either
$0_{\mathcal{C}\left(  T\right)  }\in\operatorname*{ri}\left[  g\left(
\mathbb{R}^{n}\right)  +\left(  Z\cap\mathbb{R}_{+}^{T}\right)  \right]  $ or
the SCQ holds. Then $\widehat{x}$ is a weak efficient solution of $\left(
P\right)  $ if and only if there exist $p$ nonnegative regular Borel measures
on $T,$ $\mu_{1},...,\mu_{p},$ such that there is no $x\in\mathbb{R}^{n}$
satisfying
\[
f_{i}(x)+\int\nolimits_{T}g_{t}(x)d\mu_{i}\left(  t\right)  <f(\widehat
{x}),\;\forall i\in I.
\]

\end{theorem}

\begin{proof}
According to \cite[Lemma 3.8]{DIGOLOLO}, any of the following CQs guarantees
the closedness of the set in (\ref{closed set}), and so the fulfillment of
statement $(i)$ in Lemma \ref{k4}:\newline Q1: There exists $x_{0}%
\in\mathbb{R}^{n}$ such that $g(x_{0})\in-\operatorname*{int}\left(
Z\cap\mathbb{R}_{+}^{T}\right)  .$\newline Q2: $Z$ is a Fr\'{e}chet space and
$0_{Z}\in\operatorname*{sqri}\left(  g\left(
%TCIMACRO{\dbigcap \limits_{t\in T}}%
%BeginExpansion
{\displaystyle\bigcap\limits_{t\in T}}
%EndExpansion
\operatorname{dom}g_{t}\right)  +\left(  Z\cap\mathbb{R}_{+}^{T}\right)
\right)  .$\newline So, we have just to show that at either Q1 or Q2 holds
under the assumptions of $(i)$ - $(iii)$, so that the conclusion follows from
Lemma \ref{OCA1}.\newline$(i)$ Take $Z=\mathbb{R}^{T}$ equipped with the
product topology. Observe that the topological dual $Z^{\ast}$ of $Z$\ is here
the subspace of $\mathbb{R}^{T}$ formed by the functions with finite support.
It can be realized that $Z$ is a Fr\'{e}chet space if and only if $T$ is
countable. Moreover, $y^{\ast}\circ f$ is continuous for all $y^{\ast}%
\in\mathbb{R}^{p}$, so that Q2 holds. \newline$(ii)$ Since $-\infty
<\iota\left(  x\right)  \leq g_{t}(x)\leq\psi\left(  x\right)  <+\infty$ for
all $t\in T$ and $x\in\mathbb{R}^{n},$ $g:\mathbb{R}^{n}\rightarrow Z,$ where
$Z:=\left\{  h\in\mathcal{C}\left(  T\right)  \mid h\text{ is bounded}%
\right\}  .$ Since $T$ is a normal space, its dual space $Z^{\ast}$ is formed
by the bounded finitely additive measures on $T$ \cite[Theorem IV.6.2]{DUSCH}.
Moreover, it is easy to see that $\operatorname*{int}\left(  Z\cap
\mathbb{R}_{+}^{T}\right)  =\left\{  h\in Z:\inf_{T}h>0\right\}  .$ If $x_{0}$
is a strong Slater point with associated scalar $\varepsilon$, then $\inf
_{T}\left[  -g\left(  x_{0}\right)  \right]  \geq\varepsilon>0,$ so that
$g\left(  x_{0}\right)  \in-\operatorname*{int}\left(  Z\cap\mathbb{R}_{+}%
^{T}\right)  $ and Q1 holds. \newline$(iii)$ The continuity assumption implies
that $g:\mathbb{R}^{n}\rightarrow\mathcal{C}\left(  T\right)  .$\ Take
$Z=\mathcal{C}\left(  T\right)  $ equipped with the supremum norm, whose
topological dual is formed by the regular Borel measures on $T.$ Under the
SCQ, Q1 holds by the same argument as in $(ii),$ while, taking into account
that any Banach space is Fr\'{e}chet and $\operatorname{dom}g_{t}%
=\mathbb{R}^{n}$ for all $t\in T$ in this case, $(Q2)$ is equivalent here to
$0_{Z}\in\operatorname*{ri}\left[  g\left(  \mathbb{R}^{n}\right)  +\left(
Z\cap\mathbb{R}_{+}^{T}\right)  \right]  .$
\end{proof}

We now revisit Examples \ref{Example1} and \ref{Example2}.

In Example \ref{Example1}, $(P)$ satisfies LFMCQ at $\widehat{x}=0$. Since
$F_{\ast}(\widehat{x})+G_{\ast}(\widehat{x})=\left[  -2,-1\right]
+\mathbb{R}_{+}=\left[  -2,+\infty\right)  \ni0,$ the weak KKT condition holds
at $\widehat{x}=0.$ Moreover, $\xi=\left(  -2,-1\right)  $ is fixed and
\begin{equation}
\vartheta(x,\lambda)=-\left(  2\lambda_{1}+\lambda_{2}\right)  \inf\left\{
x-y\mid y\leq0\right\}  =-\left(  2\lambda_{1}+\lambda_{2}\right)
x,\label{4.2}%
\end{equation}
so that $\vartheta(\widehat{x},\lambda)=0$ for all $\lambda\in\Delta_{+}^{2}.$
Concerning Theorem \ref{ThZ}, notice that the index set $T$\ (equipped with
the metric induced by the absolute value)\ is countable, it is normal (as it
is metric) and the functions $\left(  \cdot,x\right)  \mapsto g_{\cdot}\left(
x\right)  $ are continuous for all $x\in\mathbb{R}^{n}$ (as the topology on
$T$ is the discrete one), but the condition $0_{\mathbb{R}^{T}}\in
\operatorname*{sqri}\left[  g\left(  \mathbb{R}^{n}\right)  +\mathbb{R}%
_{+}^{T}\right]  $ can hardly be checked (as the set $g\left(  \mathbb{R}%
^{n}\right)  +\mathbb{R}_{+}^{T}$ is here the sum of the cone of nonnegative
sequences with the line in $\mathbb{R}^{T}$ which passes through the sequence
$\left\{  u_{k}\right\}  _{k=0}^{\infty}$ such that $u_{0}=0,$ $u_{2k+1}%
=-\frac{1}{k+1},$ $k=0,1,2,\ldots,$ and $u_{2k}=-\frac{1}{k},$ $k=1,2,\ldots
.,$ and is parallel to the sequence $\left\{  v_{k}\right\}  _{k=0}^{\infty}$
such that $v_{0}=2,$ $v_{t}=1$ for $t$ even, and $v_{t}=3$ for $t$ odd).
Observe that, in either case, the optimality condition in $(i)$ holds for
$\lambda^{1}\in\mathbb{R}^{T}$ such that $\lambda_{0}^{1}=1$ and $\lambda
_{t}^{1}=0$ for all $t\in\mathbb{N},$ and $\lambda^{2}\in\mathbb{R}^{T}$
arbitrary with finite , as $f_{1}+\lambda_{0}^{1}g_{0}$ is the null function
on $\mathbb{R}.$ Moreover, since $\psi$ and $\iota\left(  x\right)
=\min\left\{  x-1,3x-1\right\}  $ are real-valued and SSCQ holds, $(ii)$
applies, with $\mu_{1}$ being the atomic measure concentrating a unit mass at
$0$\ and $\mu_{2}$ arbitrary as $f_{1}\left(  x\right)  +\int\nolimits_{T}%
g_{t}(x)d\mu_{i}\left(  t\right)  =f_{1}\left(  x\right)  +g_{0}\left(
x\right)  =0$ for all $x\in\mathbb{R}.$ So, we can assert that $\widehat{x}%
$\ is a weak efficient solution for $(P)$ on the basis of Theorems \ref{qwe},
\ref{gap1} and \ref{ThZ}.

The situation is quite different at $\widehat{x}=0_{2}$ in Example
\ref{Example2}, where ACQ holds while CCCQ, LFMCQ, and the weak KKT condition
fail. This means that no conclusion on the weak efficiency of $\widehat{x}$
can be obtained from Theorems \ref{3.2} and \ref{qwe}. Concerning the gap
function, $\xi=\left(  -1,0,-1,0\right)  $ is also fixed and%
\begin{equation}
\vartheta(x,\lambda)=\left(  \lambda_{1}+\lambda_{2}\right)  \sup\left\{
y_{1}-x_{1}\mid y_{1}\leq0\right\}  =-\left(  \lambda_{1}+\lambda_{2}\right)
x_{1},\label{4.3}%
\end{equation}
so that we have again $\vartheta(\widehat{x},\lambda)=0$ for all $\lambda
\in\Delta_{+}^{2}.$ Regarding Theorem \ref{ThZ}, observe that $T$ is the same
as in Example \ref{Example1}, but neither $(ii)$ nor $(iii)$ can be applied
because SSCQ fails and $\psi$ is not real-valued, respectively, while checking
the interiority condition in $(i)$ is again a hard task. So, only Theorem
\ref{gap1} allows to conclude easily that $\widehat{x}$ is a weak efficient
solution. This example shows that we cannot replace LFMCQ by ACQ, KTCQ or WADQ
in Corollary \ref{k4}.

\section{Optimality conditions for efficiency}

We say that the \emph{strong KKT condition} holds at $\widehat{x}\in S$ when
there exist $\alpha_{i}>0$ for $i\in I$ with $\sum_{i=1}^{p}\alpha_{i}=1$, and
$\beta_{t}\geq0$ for $t\in T(\widehat{x})$, with $\beta_{t}\neq0$ for finitely
many indexes, such that%
\[
0_{n}\in\sum_{i=1}^{p}\alpha_{i}\partial f_{i}(\widehat{x})+\sum_{t\in
T(\widehat{x})}\beta_{t}\partial g_{t}(\widehat{x}).
\]
Denoting $\Delta_{++}^{p}:=\left\{  \left(  \alpha_{1},...,\alpha_{p}\right)
\in\mathbb{R}_{++}^{p}\mid\sum_{i=1}^{p}\alpha_{i}=1\right\}  , $
\cite[Theorem 6.9]{ROC} yields
\[
\operatorname*{ri}F_{\ast}\left(  \widehat{x}\right)  =%
%TCIMACRO{\dbigcup \limits_{\left(  \alpha_{1},...,\alpha_{p}\right)  \in
%\Delta_{++}^{p}}}%
%BeginExpansion
{\displaystyle\bigcup\limits_{\left(  \alpha_{1},...,\alpha_{p}\right)
\in\Delta_{++}^{p}}}
%EndExpansion
\sum_{i=1}^{p}\alpha_{i}\operatorname*{ri}\partial f_{i}(\widehat{x})\subseteq%
%TCIMACRO{\dbigcup \limits_{\left(  \alpha_{1},...,\alpha_{p}\right)  \in
%\Delta_{++}^{p}}}%
%BeginExpansion
{\displaystyle\bigcup\limits_{\left(  \alpha_{1},...,\alpha_{p}\right)
\in\Delta_{++}^{p}}}
%EndExpansion
\sum_{i=1}^{p}\alpha_{i}\partial f_{i}(\widehat{x}),
\]
and the inclusion is an equation whenever $f_{i}$ is differentiable at
$\widehat{x}$ for all $i\in I.$ Thus, $0_{n}\in\operatorname*{ri}F_{\ast
}\left(  \widehat{x}\right)  +G_{\ast}\left(  \widehat{x}\right)  $\ is a
sufficient condition for the strong KKT condition at $\widehat{x}$ and it is
also necessary when the objective functions are differentiable at $\widehat
{x}.$

\begin{theorem}
[Strong KKT necessary condition under EADQ and MOQ]\label{SKKT2}Let
$\widehat{x}$ be an efficient solution of $(P)$. If the EADQ and the MOQ hold
at $\widehat{x}$, then $(P)$ satisfies the strong KKT condition at
$\widehat{x}.$
\end{theorem}

\begin{proof}
For the sake of simplicity, we replace $Q^{l}(\widehat{x})$ by $Q^{l}$ in this
section. We can assume without loss of generality that $p\geq2.$ We present
the proof in four steps.\newline\textbf{Step 1.} We claim that
\begin{equation}
\big(\partial f_{l}(\widehat{x})\big)^{-}\cap D(Q^{l},\widehat{x}%
)=\emptyset\ ,\ \forall l\in I.\label{95301}%
\end{equation}
On the contrary, suppose that for some $l\in I$ there is a vector $d$ such
that
\begin{equation}
d\in\big(\partial f_{l}(\widehat{x})\big)^{-}\cap D(Q^{l},\widehat
{x}).\label{953}%
\end{equation}
By the definition of $D(Q^{l},\widehat{x})$, there exists a $\delta>0$ such
that $\widehat{x}+\varepsilon d\in Q^{l}$ for each $\varepsilon\in(0,\delta)$.
Thus, due to the definition of $Q^{l},$ we obtain that
\begin{equation}
\left\{
\begin{array}
[c]{ll}%
f_{i}(\widehat{x}+\varepsilon d)\leq f_{i}(\widehat{x}), & \forall i\in
I\setminus\{l\},\ \forall\varepsilon\in(0,\delta),\\
\widehat{x}+\varepsilon d\in S, & \forall\varepsilon\in(0,\delta).
\end{array}
\right. \label{751}%
\end{equation}
On the other hand, (\ref{953}) leads to $f_{l}^{\prime}(\widehat{x};d)<0$.
This means that there exists a $\delta_{l}>0$ satisfying
\begin{equation}
f_{l}(\widehat{x}+\varepsilon d)-f_{l}(\widehat{x})<0,\text{ }\forall
\varepsilon\in(0,\delta_{l}).\label{3.5}%
\end{equation}
Taking $\widehat{\delta}:=\min\{\delta,\delta_{l}\}$ and $\varepsilon
\in(0,\widehat{\delta})$, (\ref{3.5}) and (\ref{751}) contradict the
efficiency of $\widehat{x}$. Therefore, our claim holds.\newline\textbf{Step
2.} Let $\widehat{d}\in C(Q^{l},\widehat{x})=\overline{D(Q^{l},\widehat{x})}$
for some $l\in I$. Then, there exists a sequence $\{d_{k}\}_{k=1}^{\infty}$ in
$D(Q^{l},\widehat{x})$ converging to $\widehat{d}$. By (\ref{95301}) and the
continuity of $f_{l}^{\prime}(\widehat{x};.)$ we get
\[
f_{l}^{\prime}(\hat{x};\hat{d})=f_{l}^{\prime}\big(\hat{x};\lim_{n\rightarrow
\infty}d_{k}\big)=\lim_{n\rightarrow\infty}f_{l}^{\prime}(\hat{x};\hat{d}%
_{k})\geq0,
\]
so that%
\[
\big(\partial f_{l}(\widehat{x})\big)^{-}\cap C(Q^{l},\widehat{x}%
)=\emptyset\ ,\ \forall l\in I.
\]
Hence,
\begin{equation}
\Big(%
%TCIMACRO{\dbigcup \limits_{i=1}^{p}}%
%BeginExpansion
{\displaystyle\bigcup\limits_{i=1}^{p}}
%EndExpansion
\big(\partial f_{i}(\widehat{x})\big)^{-}\Big)\cap\Big(\bigcap_{i=1}%
^{p}C(Q^{i},\widehat{x})\Big)=\emptyset.\label{77889}%
\end{equation}
\newline\textbf{Step 3.} We claim that
\begin{equation}
0_{n}\in\operatorname*{ri}\big(F_{\ast}(\widehat{x})\big)+G_{\ast}(\widehat
{x}).\label{125}%
\end{equation}
On the contrary, suppose that (\ref{125}) does not hold. Then
$\operatorname*{ri}\big(F_{\ast}(\widehat{x})\big)\cap\big(-G_{\ast}%
(\widehat{x})\big)=\emptyset.$ Thus, by the proper separation theorem
(\cite[Theorem 11.3]{ROC}) and noting that $\big(-G_{\ast}(\widehat{x})\big)$
is a convex cone, it follows that there is a hyperplane $H_{d}:=\{x\in
\mathbb{R}^{n}\mid d^{\prime}x=0\}$ for some $d\in\mathbb{R}^{n}%
\setminus\{0\}$ separating $F_{\ast}(\widehat{x})$ and $\big(-G_{\ast
}(\widehat{x})\big)$ properly. In other words, there exists a vector
$d\in\mathbb{R}^{n}$ satisfying $0_{n}\neq d\in\big(F_{\ast}(\widehat
{x})\big)^{0}\cap\big(G_{\ast}(\widehat{x})\big)^{0}=F^{0}(\widehat{x})\cap
G^{0}(\widehat{x}).$ Thus, owning to EADQ and MOQ we conclude that
\[
d\in\Big(%
%TCIMACRO{\dbigcup \limits_{i=1}^{p}}%
%BeginExpansion
{\displaystyle\bigcup\limits_{i=1}^{p}}
%EndExpansion
\big(\partial f_{i}(\widehat{x})\big)^{-}\Big)\cap\Big(\bigcap_{i=1}%
^{p}C(Q^{i},\widehat{x})\Big),
\]
which contradicts (\ref{77889}).\newline\textbf{Step 4.} The result is
immediate from (\ref{125}) and the fact that (see, \cite[Theorem 6.9]{ROC})
\[
\operatorname*{ri}\big(F_{\ast}(\widehat{x})\big)\subseteq\Big \{\sum
_{i=1}^{p}\alpha_{i}\xi_{i}\mid\xi_{i}\in\partial f_{i}(\widehat{x}%
),\ \alpha_{i}>0,\ \sum_{i=1}^{p}\alpha_{i}=1\Big \}.
\]

\end{proof}

The proof of the following theorem is exactly the same as the one of Theorem
\ref{qwe}, and so has been omitted.

\begin{theorem}
[Strong KKT sufficient condition]\label{qwe2}If the strong KKT condition hods
at $\widehat{x}\in S,$ then, $\widehat{x}$ is an efficient solution of
$\left(  P\right)  .$
\end{theorem}

We now characterize the efficiency through the gap function.

\begin{theorem}
[Characterization under EADQ and MOQ via gap function]\label{gap3}Let
$\widehat{x}\in S.$ The following statements hold true:\newline$(i)$ If
$\vartheta(\widehat{x},\xi,\lambda)=0$ for some $\xi:=(\xi_{1},\ldots,\xi
_{p})\in\prod_{i=1}^{p}\Big(\partial f_{i}\left(  \widehat{x}\right)  \Big)$
and $\lambda>0_{p}$, then $\widehat{x}$ is an efficient solution for
$(P)$.\newline$(ii)$ If $\widehat{x}$ is an efficient solution for $(P)$ where
the EADQ and the MOQ hold, then there exist $\xi\in\prod_{i=1}^{p}\partial
f_{i}(\widehat{x})$ and $\lambda>0_{p}$ such that $\vartheta(\widehat{x}%
,\xi,\lambda)=0.$
\end{theorem}

\begin{proof}
$(i)$ Assume that $\vartheta(\widehat{x},\xi,\lambda)=0$, while $\widehat{x}$
is not an efficient solution for $(P)$. Then, there exist a $x^{\ast}\in S$
and an index $k\in\{1,,,,,p\}$ such that:
\[
\left\{
\begin{array}
[c]{ll}%
f_{i}(x^{\ast})\leq f_{i}(\widehat{x}),\ \ \forall i\in I, & \\
f_{k}\left(  x^{\ast}\right)  <f_{k}\left(  \widehat{x}\right)  . &
\end{array}
\right.  \
\]
Thus,
\[
\ \left\{
\begin{array}
[c]{ll}%
\xi_{i}^{\prime}(x^{\ast}-\widehat{x})\leq0,\ \ \forall i\in I, & \\
\xi_{k}^{\prime}(x^{\ast}-\widehat{x})<0. &
\end{array}
\right.
\]
From the latter inequalities and the assumption that $\lambda>0_{p}$, we get
$\sum_{i=1}^{p}\lambda_{i}\xi_{i}^{\prime}(\widehat{x}-x^{\ast})>0.$ Hence,
$\vartheta(\widehat{x},\xi,\lambda)>0$, in contradiction with the
assumption.\newline$(ii)$ The proof is exactly the same as the one of Theorem
\ref{gap1}$(ii).$
\end{proof}

In Example \ref{Example1}, the strong KKT condition holds at $\widehat{x}=0$
as
\[
\operatorname*{ri}F_{\ast}\left(  \widehat{x}\right)  +G_{\ast}\left(
\widehat{x}\right)  =\left(  -2,-1\right)  +\mathbb{R}_{+}=\left(
-2,+\infty\right)  \ni0.
\]
Moreover, $\vartheta(\widehat{x},\lambda)=0$ for all $\lambda>0_{2}$ by
(\ref{4.2}). Since the EADQ and the MOQ hold at $\widehat{x},$ we can assert
that $\widehat{x}$ is an efficient solution for $(P)$ for by Theorems
\ref{qwe2} and \ref{gap3}.

In Example \ref{Example2} we have seen that even the weak KKT condition fails
at $\widehat{x}=0_{2}$ while $\vartheta(\widehat{x},\lambda)=0$ for all
$\lambda>0_{2}$ by (\ref{4.3}). Concerning the data qualifications, EADQ holds
but MOQ fails, Thus, no conclusion on the efficiency of $\widehat{x}$\ can be
obtained from the results in this section.

\section{Optimality conditions for isolated efficiency}

We say that the \emph{perturbed KKT condition} holds at $\widehat{x}\in S$
when there exists a positive scalar $\nu$ such that, for any vector $w\in
\nu\mathbb{B}_{n},$ there exist scalars $\alpha_{i}\geq0$ for $i\in I$ with
$\sum_{i=1}^{p}\alpha_{i}=1$, a finite set $T^{\ast}\subseteq T(\widehat
{x}),\ $and corresponding scalars $\beta_{t}\geq0$ for $t\in T^{\ast},$ such
that%
\[
w\in\sum_{i=1}^{p}\alpha_{i}\partial f_{i}(\widehat{x})+\sum_{t\in T^{\ast}%
}\beta_{t}\partial g_{t}(\widehat{x}).
\]
In geometric terms, the perturbed KKT condition holds at $\widehat{x}\in S$ if
and only if
\begin{equation}
0_{n}\in\operatorname*{int}\left(  F_{\ast}(\widehat{x})+G_{\ast}(\widehat
{x})\right)  .\label{4.1}%
\end{equation}

\begin{theorem}
[Perturbed KKT necessary condition under PMFCQ]\label{isolate} Suppose that
$\widehat{x}$ is an isolated efficient solution of $(P)$ with constant $\nu>0$
such that all constraint functions are continuously differentiable around
$\widehat{x}.$ Then, the following statements hold true:\newline$(i)$ If the
PMFCQ holds at $\widehat{x}$, then
\begin{equation}
\nu\mathbb{B}_{n}\subseteq F_{\ast}(\widehat{x})+\bigcap_{\varepsilon
>0}\overline{\operatorname*{cone}}\big(\{\nabla g_{t}(\widehat{x})\mid t\in
T_{\varepsilon}(\widehat{x})\}\big).\label{q1210}%
\end{equation}
\newline$(ii)$ If $(P)$ is continuous, and the MFCQ holds at $\widehat{x}$,
then the perturbed KKT condition holds at $\widehat{x}.$
\end{theorem}

\begin{proof}
$(i)$ For each $x\in S$ we consider the DC (difference of convex) function
$\varphi(x):=\max_{i\in I}\{f_{i}(x)-f_{i}(\widehat{x})\}-\nu||x-\widehat
{x}||$. The definition of isolated efficiency means that $\varphi(x)\geq0$ for
all $x\in S$. Since $\varphi(\widehat{x})=0$, then $\widehat{x}$ is a
minimizer of scalar optimization problem $\min_{x\in S}\varphi(x)$. Thus,
using \cite[Corollary 1]{TIPEAI}, we obtain
\begin{equation}
\partial\left(  \nu||\cdot-\widehat{x}||\right)  (\widehat{x})\subseteq
\partial\big(\max_{i\in I}\{f_{i}(\cdot)-f_{i}(\widehat{x})\}\big)(\widehat
{x})+N(S,\widehat{x}).\label{is1}%
\end{equation}
On the other hand, from the well known rules of subdifferential calculus (see,
e.g., \cite[Section VI.4]{HirLem}) one has
\begin{equation}
\partial\left(  \nu||\cdot-\widehat{x}||\right)  (\widehat{x})=\nu
\mathbb{B}_{n}\ \ \ \text{and}\ \ \ \ \partial\big(\max_{i\in I}%
\{f_{i}(.)-f_{i}(\widehat{x})\}\big)(\widehat{x})\subseteq F_{\ast}%
(\widehat{x}).\label{is2}%
\end{equation}
Combining (\ref{is1}), (\ref{is2}), the PMFCQ, and \cite[Proposition 1 and
Theorem 1]{MORDNG}, we get (\ref{q1210}).\newline$(ii)$ Observe that the
continuity of $(P),$ together with \cite[Proposition 2]{MORDNG}, implies the
equivalence between the PMFCQ and the MFCQ. Now, invoking \cite[Corollary
2]{MORDNG} and taking (\ref{q1210}) into account, we get $\nu\mathbb{B}%
_{n}\subseteq\left(  F_{\ast}(\widehat{x})+G_{\ast}(\widehat{x})\right)  ,$ so
that (\ref{4.1}) holds. Hence, the perturbed KKT condition holds at
$\widehat{x}.$
\end{proof}

\begin{theorem}
[Perturbed KKT sufficient condition]\label{isolate1}If the perturbed KKT
condition holds at $\widehat{x}\in S,$ then, $\widehat{x}$ is an isolated
efficient solution for $(P)$.
\end{theorem}

\begin{proof}
Let $x\in S\diagdown\{\widehat{x}\}$. Since the perturbed KKT condition holds
at $\widehat{x},$ there exists $\nu>0$ such that $\nu\mathbb{B}_{n}\subseteq$
$\operatorname*{int}\left(  F_{\ast}(\widehat{x})+G_{\ast}(\widehat
{x})\right)  .$ Recalling (\ref{4.1}), and since $\frac{x-\widehat{x}%
}{||x-\widehat{x}||}\in\mathbb{B}_{n}$, there exist $\alpha_{i}\geq0$ and
$\xi_{i}\in\partial f_{i}(\widehat{x}),$ for $i\in I$ with $\sum_{i=1}%
^{p}\alpha_{i}=1$, and $T^{\ast}\subseteq T(\widehat{x})$ with $|T^{\ast
}|<\infty$, $\zeta_{t}\in\partial g_{t}(\widehat{x})$\ and $\beta{_{t}}\geq0,$
for $t\in T^{\ast},$ such that
\begin{equation}
\nu\frac{x-\widehat{x}}{||x-\widehat{x}||}=\sum_{i=1}^{p}\alpha_{i}\xi
_{i}+\sum_{t\in T^{\ast}}\beta{_{t}}\zeta_{t}(\widehat{x}).\label{5.1}%
\end{equation}
Multiplying both members of (\ref{5.1}) by $x-\widehat{x}$, one gets
\begin{align*}
\nu||x-\widehat{x}||  & =\sum_{i=1}^{p}\alpha_{i}\xi_{i}^{\prime}%
(x-\widehat{x})+\sum_{t\in T^{\ast}}\beta{_{t}}\zeta_{t}^{\prime}%
(x-\widehat{x})\\
& \leq\sum_{i=1}^{p}\alpha_{i}\big(f_{i}(x)-f_{i}(\widehat{x})\big)+\sum_{t\in
T^{\ast}}\beta{_{t}}\big(g_{t}(x)-g_{t}(\widehat{x})\big)\\
& \leq\sum_{i=1}^{p}\alpha_{i}\big(f_{i}(x)-f_{i}(\widehat{x})\big)\\
& \leq\max_{i\in I}\big\{f_{i}(x)-f_{i}(\widehat{x})\big\},
\end{align*}
where the last inequality holds as $\sum_{i=1}^{p}\alpha_{i}=1$. This shows
$\widehat{x}$ is an isolated efficient solution for $(P)$.
\end{proof}

In Example \ref{Example1}, $F_{\ast}(\widehat{x})+G_{\ast}(\widehat
{x})=\left[  -2,-1\right]  +\mathbb{R}_{+}=\left[  -2,+\infty\right)  ,$ which
is certainly a neighborhood of $0,$ so that the perturbed KKT condition holds
at $\widehat{x}$ with constant $\nu=2.$ Then, $\widehat{x}$ is an isolated
efficient solution by Theorem \ref{isolate1}. In fact, $\max_{i\in I}%
\{f_{i}(x)-f_{i}(\widehat{x})\}=-2x\geq2\left\vert x\right\vert $ for all
$x\in S=-\mathbb{R}_{+}.$

In Example \ref{Example2}, $F_{\ast}(\widehat{x})+G_{\ast}(\widehat
{x})=\big\{x\in\mathbb{R}^{2}\mid x_{1}\geq-1,\ x_{2}>0\big\}\cup\{\left(
-1,0\right)  \},$ which is not a neighborhood of $\widehat{x}=0_{2}.$ Thus,
the perturbed KKT condition fails at $\widehat{x},$ but we cannot conclude
from Theorem \ref{isolate} that $\widehat{x}$ is not an isolated efficient
solution as $\left(  P\right)  $ is not continuous and MFCQ fails, i.e., we
are in a dubious case. Actually, $\max_{i\in I}\{f_{i}(x)-f_{i}(\widehat
{x})\}=-x_{1}\geq\left\Vert x\right\Vert $ for all $x\in S=-\mathbb{R}_{+}%
^{2}, $ i.e., $\widehat{x}$ is an isolated efficient solution too.

The final result in this paper requires $\left(  P\right)  $ be continuous, so
that it cannot be applied to the above examples.

\begin{theorem}
[Characterization under MFCQ via gap function]Let $\left(  P\right)  $ be a
continuous problem and $\widehat{x}\in S$ be such that all constraint
functions are continuously differentiable around $\widehat{x}$ and MFCQ holds.
Then, $\widehat{x}$ is an isolated efficient solution if and only if there
exists $\nu>0$ such that for all $w\in\nu\mathbb{B}_{n}$ there exist $\xi
\in\prod_{i=1}^{p}\partial f_{i}(\widehat{x})$ and $\lambda\geq0_{p}$ such
that $\vartheta(\widehat{x},\xi-w,\lambda)=0.$
\end{theorem}

\begin{proof}
Consider the parametric convex MOSIO problem, with parameter $w\in
\mathbb{R}^{n},$%
\[
(P_{w})\ \ \ \ \text{minimize}\ \left(  f_{1}(x)-w,\ldots,f_{p}(x)-w\right)
\ \text{subject to}\ g_{t}(x)\leq0,\ t\in T,
\]
whose gap function is $\vartheta_{w}(x,\eta,\lambda)=\sup_{y\in S}%
\Big\{\sum_{i=1}^{p}\lambda_{i}\eta_{i}^{\prime}(x-y)\Big\},$ with $\eta
_{i}\in\partial\left(  f_{i}-w\right)  (\widehat{x})=\partial f_{i}%
(\widehat{x})-w.$ The change of variables $\xi_{i}=\eta_{i}+w,$ with $\xi
_{i}\in\partial f_{i}(\widehat{x}),$ $i\in I,$ yields the identity
$\vartheta_{w}(x,\eta,\lambda)=\vartheta(\widehat{x},\xi-w,\lambda).$We claim
the equivalence between the following statements:\newline$(i)$ $\widehat{x}$
is an isolated efficient solution of $\left(  P\right)  .$\newline$(ii)$ For
any $w\in\nu\mathbb{B}_{n},$ there exist scalars $\alpha_{i}\geq0$ for $i\in
I$ with $\sum_{i=1}^{p}\alpha_{i}=1$, a finite set $T^{\ast}\subseteq
T(\widehat{x}),\ $and corresponding scalars $\beta_{t}\geq0$ for $t\in
T^{\ast},$ such that%
\[
0_{n}\in\sum_{i=1}^{p}\alpha_{i}\partial\left(  f_{i}-w\right)  (\widehat
{x})+\sum_{t\in T^{\ast}}\beta_{t}\partial g_{t}(\widehat{x}).
\]
\newline$(iii)$ $\widehat{x}$ is a weak efficient solution of $\left(
P_{w}\right)  $ for all $w\in\nu\mathbb{B}_{n}.$\newline$(iv)$ For any
$w\in\nu\mathbb{B}_{n},$ there exist $\xi\in\prod_{i=1}^{p}\partial
f_{i}\left(  x\right)  $ and $\lambda\in\Delta_{+}^{p}$ such that
$\vartheta(\widehat{x},\xi-w,\lambda)=0.$\newline\newline[$(i)\Leftrightarrow
(ii)$] It follows from Theorems \ref{isolate} and \ref{isolate1} as $(ii)$ is
a reformulation of the perturbed KKT condition for $\left(  P\right)
.$[$(ii)\Leftrightarrow(iii)$] It follows from Corollary \ref{k4}, applied to
$(P_{w}),$ taking into account that LFMCQ holds at $\widehat{x}$ by Theorem
\ref{ThCQs}\ (the constraints of $(P_{w})$ and $(P)$ coincide).\ \newline%
[$(iii)\Leftrightarrow(iv)$] It follows from Theorem \ref{gap1}, applied to
$\left(  P_{w}\right)  .$
\end{proof}

\end{document}